\documentclass[seceqn,aop,MSNbibl,amsmath,v1.1,preprint]{imsart-mod}
\pdfoutput=1

\usepackage{mathrsfs}
\usepackage{amsthm,amssymb,graphicx,url}
\usepackage{mathptmx}
\usepackage{cite}
\usepackage{microtype}

\doi{https://doi.org/10.1214/18-AOP1320}
\volume{47}
\issue{5}
\pubyear{2019,}
\firstpage{2754}
\lastpage{2800}

\startlocaldefs
\newcommand{\rrvert}{\vert}
\newcommand{\rrVert}{\Vert}
\newcommand{\llvert}{\vert}
\newcommand{\llVert}{\Vert}
\newcommand{\arxivurl}[1]{arXiv:#1}

\newtheorem{theorem}{Theorem}[section]

\newtheorem{lem}[theorem]{Lemma}
\newtheorem{cor}[theorem]{Corollary}
\theoremstyle{definition}
\newtheorem{defn}[theorem]{Definition}
\newtheorem{remark}[theorem]{Remark}

\numberwithin{equation}{section}

\newcommand{\supp}{\operatorname{supp}}
\newcommand{\wt}{\widetilde}
\newcommand{\mcl}{\mathcal}
\newcommand{\W}{\mcl{W}}
\newcommand{\defnphrase}[1]{\emph{#1}}
\newcommand{\st}{ : }
\newcommand{\Reals}{\mathbb{R}}
\newcommand{\Nats}{\mathbb{N}}
\newcommand{\edges}{e}
\newcommand{\vertices}{v}
\newcommand{\loops}{\ell}
\newcommand{\EE}{\mathbb{E}}
\newcommand{\convPr}{\xrightarrow{ p }}
\newcommand{\convDist}{\xrightarrow{ d }}
\newcommand{\equaldist}{\overset{d}{=}}
\newcommand{\upto}{ \uparrow }
\newcommand{\given}{\mid}
\newcommand{\binDist}{\operatorname{Bin}}
\newcommand{\uniDist}{\operatorname{Uni}}
\newcommand{\poiDist}{\operatorname{Poi}}
\newcommand{\BernDist}{\operatorname{Bern}}
\newcommand{\as}{\textrm{ a.s.}}
\newcommand{\equalas}{\overset{\mathrm{a.s.}}{=}}
\newcommand{\intd}{\mathrm{d}}
\newcommand{\distiid}{\overset{\mathrm{iid}}{\ \sim\ }}
\newcommand{\PP}{\Pi}
\newcommand{\PPDist}{\mathrm{PP}}
\newcommand{\exclude}{\setminus}
\newcommand{\Lebesgue}{\Lambda}
\newcommand{\StarF}{S}\newcommand{\IsoF}{I}
\newcommand{\E}{\mathbb{E}}
\newcommand{\R}{\mathbb{R}}
\newcommand{\eps}{\varepsilon}
\newcommand{\KEG}{\Gamma}
\newcommand{\KEGDinf}[1]{\mathrm{KEG}(#1)}
\newcommand{\GS}[1]{\mathscr{G}(#1)}
\newcommand{\labelLim}{\zeta}
\newcommand{\limspace}{\mathscr{W}}
\newcommand{\graphspace}{\mathscr{G}}
\endlocaldefs

\begin{document}

\begin{frontmatter}

\title{Sampling perspectives on sparse exchangeable graphs}
\runtitle{Sampling Perspectives on Sparse Exchangeable Graphs}

\begin{aug}

\author{\fnms{Christian}~\snm{Borgs}\thanksref{m1}\ead[label=e1]{borgs@microsoft.com}},
\author{\fnms{Jennifer T.}~\snm{Chayes}\thanksref{m1}\ead
[label=e2]{jchayes@microsoft.com}},
\author{\fnms{Henry}~\snm{Cohn}\thanksref{m1}\ead
[label=e3]{cohn@microsoft.com}} \and \\
\author{\fnms{Victor}~\snm{Veitch}\thanksref{m2,t1}\ead
[label=e4]{vcv2109@columbia.edu}} \runauthor{Borgs, Chayes, Cohn and Veitch}

\affiliation{Microsoft Research\thanksmark{m1} and Columbia
University\thanksmark{m2}}

\address{C. Borgs\\
J. T. Chayes\\
H. Cohn\\
Microsoft Research\\
One Memorial Drive\\
Cambridge, Massachusetts 02142\\
USA\\\printead{e1}
\\\phantom{E-mail:} \printead*{e2}
\\\phantom{E-mail:} \printead*{e3}}
\address{V. Veitch\\
Department of Statistics\\
Columbia University\\
1255 Amsterdam Ave\\
New York, New York 10027\\
USA\\\printead{e4}}

\end{aug}

\def\thethanks{1}
\usethankscounter{thanks}
\thankstext{t1}{Work partially completed while at University of Toronto.
Supported in part by U.S. Air Force Office of Scientific Research Grant
\#FA9550-15-1-0074 and an internship at Microsoft Research New England.}

\received{\smonth{8} \syear{2017}}
\revised{\smonth{6} \syear{2018}}

\begin{abstract}
Recent work has introduced sparse exchangeable graphs and the associ\-ated
graphex framework, as a generalization of dense exchangeable graphs and the
associated graphon framework. The development of this subject involves the
interplay between the statistical modeling of network data, the
\mbox{theory} of large graph limits, exchangeability and network sampling.
The pur\-pose of the present paper is to clarify the relationships between
these subjects by explaining each in terms of a certain natural sampling
scheme associated with the graphex model. The first main technical
contribution is the introduc\-tion of \emph{sampling convergence}, a new
notion of graph limit that generalizes left convergence so that it becomes
meaningful for the sparse graph regime. The second main technical
contribution is the demonstration that the (somewhat cryptic) notion of
exchangeability underpinning the graphex framework is equivalent to a more
natural probabilistic invariance expressed in terms of the sampling scheme.
\end{abstract}

\begin{keyword}[class=AMS]
\kwd[Primary ]{60C05}
\kwd[; secondary ]{62D05} \kwd{62G09} \kwd{62G05}.
\end{keyword}

\begin{keyword}
\kwd{Network analysis} \kwd{sampling} \kwd{graph limits} \kwd{nonparametric
estimation}.
\end{keyword}
\end{frontmatter}

\setcounter{page}{1}

\section{Introduction}\label{sec1}

\stepcounter{footnote}

The present paper is concerned with the theory of graph limits, the
statistical modeling of networks and the relationship between these topics
and exchangeability. In the setting of dense graphs, these topics meet in the
theory of graphons, which are fundamental in the study of graph limits \cite
{Borgs:Chayes:Lovasz:Sos:Vesztergombi:2006:1,Lovasz:Szegedy:2006,Lovasz:Szegedy:2007,Borgs:Chayes:Lovasz:Sos:Vesztergombi:2008,Borgs:Chayes:Lovasz:Sos:Vesztergombi:2012}
(see \cite{Lovasz:2013:A} for a review) and provide the foundation for many
of the statistical \mbox{network} models in current use \cite
{Nowicki:Snijders:2001:A,Hoff:Raftery:Handcock:2002,Airoldi:Blei:Fienberg:Xing:2008,Miller:Jordan:Griffiths:2009,Lloyd:Orbanz:Ghahramani:Roy:2012}
(see \cite{Orbanz:Roy:2015} for a review). Motivated by the importance of
graphons in the dense graph setting,  \mbox{a recent series of papers} \cite
{Caron:Fox:2014,Herlau:Schmidt:Morup:2015,Veitch:Roy:2015,Borgs:Chayes:Cohn:Holden:2016,Veitch:Roy:2016,Todeschini:Caron:2016,Janson:2017}
has developed a generalization of the graphon framework to the regime of
sparse graphs, both as a tool for statistical network modeling
\cite{Veitch:Roy:2015,Borgs:Chayes:Cohn:Holden:2016} and estimation
\cite{Veitch:Roy:2016}, and as the central element of a limit theory for
large graphs
\cite{Borgs:Chayes:Cohn:Holden:2016} (see also
\cite{Janson:2016}). This generalization is compelling in that it preserves
many of the desirable properties of the graphon framework, while
simultaneously allowing much greater flexibility. However, there are some
significant interpretational issues remaining. For example, it is unclear
which real-world processes are appropriately modeled by the statistical
network models of the new framework, or how best to
characterize the properties of large graphs that are well
approximated by the new \mbox{limit} theory. The root of these difficulties
is that the new framework is derived us-ing a cryptic construction that
represents random graphs as point processes on $\Reals_{+}^2$, and then
formalizes the models of the generalized framework as those corresponding to
point processes that are exchangeable.

In the dense setting, graphons as stochastic network models can be arrived at
in at least two different ways. The first approach is simply to posit them
directly. Graphon models are the class of generative models for random graphs
in which each vertex $i$ is assigned some independent latent features $x_i$,
and conditional on these latent features, each pair of vertices $i,j$ is
connected by an edge independently with probability $W(x_i,x_j)$ determined
by the latent features of $i$ and $j$. This is a very natural class of
models, and models of this type, such as stochastic block models and latent
feature models, have a long history in the statistical networks
literature. The second approach proceeds by identifying a projective family
$(G_n)_{n \in \Nats}$ of random graphs with the upper left $n \times n$
submatrices of an infinite random adjacency matrix $A$, and then defining the
class of models to be those such that the distribution of $A$ is
invariant under joint permutations of its rows and columns. This
\emph{exchangeability} of $A$ is a natural formalization of the requirement
that the labels of the vertices of a random graph should be uninformative
about the structure of the graph. The fact that the graphon models are the
models defined by exchangeability of the infinite adjacency matrix is,
essentially, the content of the celebrated Aldous--Hoover theorem
\cite{Aldous:1981,Hoover:1979}.

Graphons as limit objects for dense graphs sequence also arise very naturally
in the dense setting: many natural notions of similarity, such as left
convergence moti\-vated by extremal graph theory, right convergence motivated
by studying statistical physics (or, equivalently, graphical) models on
graphs, as well as quotient convergence motivated by combinatorial
optimization, all lead to graphons over probabil\-ity spaces as the
completion of the space of dense graphs \cite
{Borgs:Chayes:Lovasz:Sos:Vesztergombi:2006:1,Borgs:Chayes:Lovasz:Sos:Vesztergombi:2008,Borgs:Chayes:Lovasz:Sos:Vesztergombi:2012}.
These notions of convergence turn out to all be equivalent, and can be
metrized by the \emph{cut metric} (discussed below), making the theory of
graph convergence \mbox{a well rounded} math-ematical theory. Finally,
exchangeable random graphs generated from a graphon can be shown to converge
to the generating graphon \cite{Lovasz:Szegedy:2006}, creating a first
connec\-tion between graphons as models for exchangeable random graphs and as
limits of sequences of sparse graphs. See \cite{Diaconis:Janson:2007} for a
systematic overview of the relationship between the theory of graph
convergence and the theory of exchangeable random graphs in the dense graph
setting.

The key ingredient of the generalization from the dense graph setting to the
sparse graph setting is a novel notion of exchangeability for random graphs.
In the generalized theory, the vertices of the random graphs are labeled in
$\Reals_{+}$, the edge sets of these graphs are represented as point
processes on $\Reals_{+}^2$ and invariance under vertex relabeling is encoded
as joint exchangeability of the point process. This rather abstruse
formalization was introduced as an ad hoc solution to the problem that the
more obvious notion of exchangeability implies that the corresponding random
graphs are almost surely dense. Nevertheless, the resulting models retain the
essential character of the dense graphon models: each vertex $i$ has latent
feature $x_i$ and, conditional on these latent features, each edge is
included independently with a probability determined by the latent features
of its endpoints. The essential difference is that the latent features are
now generated as a Poisson process on a $\sigma$-finite space, rather than
independently. The appeal of these models is then their close analogy to the
dense graphon models, in combination with their greater flexibility.

However, this picture is somewhat superficial, since it leaves many questions
unanswered. Why do we represent graphs as point processes? Why does the
corresponding notion of exchangeability give a much broader class of models
than the adjacency matrix exchangeability? Why should the points in the
latent feature space be distributed according to a Poisson process? What
motivates the particular way of embedding graphs into the space of graphons
over $\Reals_{+}$ that \cite{Borgs:Chayes:Cohn:Holden:2016} uses to translate
convergence in the cut metric for graphons into a notion of convergence in
metric for graphs? Why are graph limits and statistical network modeling so
closely tied together? The contribution of the present paper is to resolve
these conceptual difficulties by relating the core ideas---graph limits,
statistical network modeling and exchangeability---to a certain natural
scheme for sampling random subgraphs from larger graphs.

Our first main contribution is the introduction and development of
\defnphrase{sampling convergence}, a new notion of graph limit that
generalizes left convergence \cite
{Borgs:Chayes:Lovasz:Sos:Vesztergombi:2006:1,Borgs:Chayes:Lovasz:Sos:Vesztergombi:2008},
a core concept in the graphon theory of limits of dense graphs, to a notion
that is also meaningful for sparse graphs. We show that sampling convergence
both generalizes the metric convergence of
\cite{Borgs:Chayes:Cohn:Holden:2016} and allows us to formalize the notion of
sampling a data set from an infinite size population network; it thereby
connects graph limits and statistical network modeling. Our second main
contribution is that the ad hoc assumption of exchangeability may be replaced
by a more natural equivalent invariance given in terms of the sampling
scheme. This symmetry makes no reference to the point process representation
of random graphs or to the associated notion of exchangeability; this allows
us to understand these ideas as mathematical artifices rather than conceptual
cornerstones of the theory.

We begin by explaining our limit theory as a natural generalization of the
dense graph limit theory. In the setting of dense graphs, one of the core
limit notions is left convergence, the convergence of subgraph densities. In
the course of explaining the connection between exchangeability and graph
limits in the dense graph setting, Diaconis and Janson
\cite{Diaconis:Janson:2007} present the following perspective on left
convergence. Given a graph $G_j$, for each $k \in\Nats$ we draw a random
subgraph $H_{j,k}$ of $G_j$ by selecting $k$ vertices independently at random
and returning the induced subgraph; a sequence $G_1,G_2,\dots$ is left
convergent when, for all $k \in\Nats$, the random graphs $H_{j,k}$ converge
in distribution as $j \to\infty$. Intuitively speaking, this notion of
convergence encodes the idea that two large graphs are similar when it is
difficult to tell them apart by randomly sampling small subgraphs from each.

It is straightforward to see why left convergence is informative only for
dense graph sequences: if the graph sequence $G_1,G_2,\dots$ is sparse then
the probability that a random $k$ vertex subgraph of $G_j$ contains even a
single edge goes to $0$ as $j$ becomes large. The resolution we propose here
is, intuitively speaking, to generalize this sampling scheme in a way that
fixes the target number of \emph{edges} in the randomly sampled subgraph,
instead of the number of vertices.

The first key idea in formalizing this is the following notion for sampling
from a graph, introduced in \cite{Veitch:Roy:2016}. Here, a vertex in a
subgraph of a given graph $G$ is called \defnphrase{isolated} if it is not
contained in any edge (regardless of whether this edge is a loop edge or a
nonloop edge) of the subgraph.

\begin{defn}
A \defnphrase{$p$-sampling} $\mathsf{Smpl} (G,p)$ of a graph\footnote
{Throughout this paper, a graph will be a graph without multiple edges, but
it may not be simple; that is, it may contain edges joining a vertex to
itself. Unless explicitly mentioned, all graphs will be finite.} $G$ is a
random subgraph of $G$ given by including each vertex of $G$ independently
with probability $\min(p,1)$, then discarding all isolated vertices in the
resulting induced subgraph, and finally returning the unlabeled graph
corresponding to this subgraph.
\end{defn}

The critical property that distinguishes $p$-sampling from independent vertex
sampling is that vertices that do not participate in any edges in the vertex
induced subgraph are thrown away. Note that by definition, $\mathsf {Smpl}
(G,p)$ is always unlabeled, whether $G$ is labeled or not.

We may now define our notion of graph limit. Let $\edges(G)$ denote the
number of non-loop edges of a graph $G$.

\begin{defn}
A sequence of graphs $G_1,G_2,\dots$ is \defnphrase{sampling convergent} if,
for all $r \in\Reals_{+}$, the random graphs $\mathsf{Smpl}
(G_j,{r}/{\sqrt{2\edges(G_j)}})$ induced by
${r}/{\sqrt{2\edges(G_j)}}$-sampling of $G_j$ converge in distribution as $j
\to\infty$.
\end{defn}

For the remainder of the introduction, we will restrict our attention to
sequences of simple graphs; loops are treated in the body of the paper.

Sampling convergence can be understood as a modification of left convergence
as follows: we draw an increasing number of vertices as $j\to\infty$ because
if we drew only a fixed number $k$ then the induced graph would be empty in
the limit. Since the number of sampled vertices diverges, we instead fix the
target number of sampled edges. Because we are selecting vertices at random,
the number of edges in the vertex induced subgraph must be random, so a
natural way to fix the size of the sampled subgraph as $j \to\infty$ is to
require the expected number of edges to be constant. This requirement
dictates that each vertex is included with probability proportional to
$1/\sqrt{\edges(G_j)}$; the convention we choose for the proportionality
constant gives
\[
\EE \bigl[\edges \bigl(\mathsf{Smpl} \bigl(G_j,{r}/{\sqrt{2
\edges(G_j)}}\bigr) \bigr) \bigr] = r^2/2
\]
for all $j\in\Nats$. Because the number of sampled vertices goes to infinity
as $j \to\infty$, it is not possible to have convergence in distribution of
the vertex sampled subgraphs. This problem is solved by using $p$-sampling
instead of independent vertex sampling; that is, we simply throw away the
vertices that are isolated in the sampled subgraph.

Our first main result is that the natural limit object of a sampling
convergent sequence is a triple $\W= (I,S,W)$, where $I \in\Reals_{+}$,
$S\colon\Reals_{+}\to\Reals_{+}$ is an integrable function, and the
\emph{graphon} $W\colon\Reals_{+}^2 \to[0,1]$ is a symmetric integrable
function. This object is the (integrable) \defnphrase{graphex} at the heart
of the (sparse) exchangeable graph models. Each graphex defines a
\defnphrase{graphex process} (or
\defnphrase{Kallenberg Exchangeable Graph} in the language of
\cite{Veitch:Roy:2015,Veitch:Roy:2016}), a family of growing random graphs
$(\KEG_s)_{s\in\Reals_{+}}$ with vertices labeled in $\Reals_{+}$. Following
\cite{Borgs:Chayes:Cohn:Holden:2016}, we refer to the label of a vertex as
its \defnphrase{birth {time}}, and to $\KEG_s$ as the
\defnphrase{graphex process} at {time} $s$. For a finite labeled graph
$\KEG_s$, we denote the associated unlabeled graph by $\mathcal{G}
 (\KEG _s ) $. The
sense in which the graphex is the natural limit object is given by
Theorem~\ref{lim_is_det_int_graphex}: for every sampling convergent sequence
$G_1,G_2,\dots$ there is some integrable graphex $\W$ such that, for all
$s\in\Reals_{+}$, $\mathsf{Smpl} (G_j,{s}/{\sqrt{2\edges(G_j)}}) \convDist
\mathcal{G}  (\KEG_s ) $ as $j \to\infty$, where $(\KEG _s)_{s\in\Reals_{+}}$
is generated by $\W$. That is, the limiting distribution of the sampled
subgraph is characterized by the graphex that is the sampling convergent
limit. In this case, we say that $G_j$ is \emph{sampling convergent to $\W$}.

We complete the limit theory by showing that every integrable graphex arises
as the sampling convergent limit of some graph sequence, at least up to
certain equivalencies (Theorem~\ref{graphex_proc_samp_conv}), and by
metrizing the convergence and characterizing the associated metric space
(Theorems~\ref{theorem:met_space_struct} and \ref{theorem:compact}). In
consequence of the former result, the (integrable) graphex process models can
be understood conceptually as originating as the limit objects of sampling
convergence, without any direct appeal to exchangeability (although in fact
our technical arguments lean heavily on exchangeability and the associated
machinery).

This last observation raises the question of whether the graphex processes
can be characterized directly in terms of $p$-sampling, without appeal to
either exchangeability or graph limits. The motivation in
\cite{Veitch:Roy:2016} for the introduction of $p$-sampling was the
observation that a $p$-sampling of $\mathcal{G}  (\KEG_s ) $ is equal in
distribution to $\mathcal{G}  (\KEG_{p s} ) $; that is, this is the sampling
scheme that describes the relationship between graphex process graphs at
different {time}s. We prove in Theorem~\ref{psamp_defining} that this is in
fact a defining property of the graphex process. That is, if
$(G_s)_{s\in\Reals_{+}}$ is a family of unlabeled random graphs such that for
all $s \in\Reals_{+}$ and all $p \in (0,1)$ the $p$-sampling of $G_s$ is
equal in distribution to $G_{p s}$, then there is some graphex $\W$ such that
$G_s \equaldist\mathcal{G}
 (\KEG_{s} ) $
for all $s\in\Reals_{+}$, where $(\KEG_s)_{s\in\Reals_{+}}$ is generated by
$\W$. This gives a formal sense in which this sampling invariance is
equivalent to the notion of exchangeability originally used to define
exchangeable random graphs.

We now turn to explaining the connection between our results and statistical
network modeling, and the relationship to other notions of graph limits.

\subsection{Statistical network modeling}

The major motivation in \cite{Veitch:Roy:2015} for the introduction of
graphex process models was as a tool for the statistical analysis of
network-valued data sets. These models are attractive for this purpose
because they offer a sparse graph generalization of the graphon model and the
exchangeable array framework, which underlie many popular models. In this
setting, the conceptual challenge brought on by exchangeability is that
because it is unclear what the symmetry means in practical terms it is also
unclear what the practical applicability of the models is. In particular, we
would like a clear articulation of the circumstances under which it is
appropriate to model a data set by a graphex process.

Following \cite{Crane:Dempsey:2016:snm}, a statistical model can be
understood as consisting of two parts: a data generating process and a
sampling scheme for collecting a data set from a realization of this process.
In the network setting, this is envisioned as some real world process that
generates a large population graph from which the data set is then somehow
sampled. In order to assess the applicability of a statistical network model,
we should articulate the associated data generation mechanism and sampling
scheme.

The most obvious sampling scheme to associate with the graphex process model
is $p$-sampling. Having assumed $p$-sampling, the question of what data
generating mechanism gives rise to the population is subtle. One obvious
guiding principle is that we ought to be able to make meaningful inferences
about the population on the basis of the sample. For example, if the data
generating process is itself a graphex process with graphex $\W$ then the
sample will be distributed as finite graph generated by $\W$; inferences
about the population then take the form of inferences about~$\W$. However,
the graphex process has some properties that are highly undesirable for a
model of a data generating process. For example, a graphex process can only
grow and, moreover, can grow only by adding edges connecting to vertices that
have never been seen before. As a model for a social network this would mean
that two people who are friends may never stop being friends, and two people
who are not yet friends may never form a link in the future.

In classical statistics, data sets are often envisioned as being drawn
independently from some very large population, often idealized as infinite.
In our setting, the analogous thing is to envision a particular (fixed size)
observation as a draw from a very large population network where each vertex
is included independently with small probability. To formalize the
infinite-size population idealization, consider the limit where the size of
the population, created according to the data generating mechanism, becomes
infinite while the vertex inclusion probability goes to $0$ at a rate that
keeps the size of the observed data set constant. That is, we imagine
$\edges(G_j) \to\infty$ and the inclusion probability $p_j = \Theta
({1}/{\sqrt{\edges(G_j)}} )\mathclose {}$. In this case, a minimal
requirement for the sampled data set to be informative about the limiting
population is that the distribution of the sample should converge. We have
thus been led to the following precept: the data generating mechanism should
give rise to a sequence of population graphs that is sampling convergent.
This is as far as we need go: by Theorem~\ref{lim_is_det_int_graphex}, the
requirement of sampling convergence already implies that the observation is
distributed according to some integrable graphex~$\W$.

The preceding can be summarized as follows:
\begin{quote}
Finite size graphex processes approximate statistical network models that
arise from vertex sampling of a population that is generated according to
some sampling convergent data generating process. In the infinite
population limit, this approximation becomes exact.
\end{quote}

It is worth emphasizing that this is much broader than it may appear at first
glance. For example, this perspective may even be appropriate in situations
where we observe the entire available network, as long as the physical
mechanism generating the network is sampling convergent and the process that
restricts to a finite size observation can be modeled approximately as an
independent sampling of the vertices.

In lectures and as yet unpublished work, P. Orbanz has given a treatment of
the broad idea of defining schemes for statistical network modeling by way of
defining a sampling scheme and studying the models compatible with the
symmetries thereby induced. One perspective on the present paper is that we
work out the realization of this program for $p$-sampling.

\subsection{Graph limits}

Sampling convergence gives a notion of graph limit for deterministic
sequences of unlabeled graphs. We now explain the connection to several other
notions of large graph limit, namely:
\begin{enumerate}
\item[1.] the convergence of sequences of randomly labeled graphs,
\item[2.] the metric convergence of \cite{Borgs:Chayes:Cohn:Holden:2016}, and
\item[3.] the consistent estimation of \cite{Veitch:Roy:2016}.
\end{enumerate}

\subsubsection{Randomly labeled graphs}

The first of these is fundamental to the development of the theory in the
present paper. Exchangeability is a concept of infinite size labeled random
graphs, but the theory of graph limits deals with nonrandom sequences of
graphs. It is then somewhat mysterious why there should be such a close
connection between graph limits and exchangeable random graphs.\looseness=1

In the dense graph setting, this manifested as the development of the theory
of exchangeable arrays \cite{Aldous:1981,Hoover:1979,Kallenberg:2005} on one
hand and the independent development of the theory of dense graph limits
\cite{Borgs:Chayes:Lovasz:Sos:Vesztergombi:2006:1,Lovasz:Szegedy:2006,Lovasz:Szegedy:2007,Borgs:Chayes:Lovasz:Sos:Vesztergombi:2008,Borgs:Chayes:Lovasz:Sos:Vesztergombi:2012}
on the other. The connection between the two perspectives is explained by
\cite{Diaconis:Janson:2007,Austin:2008}, the development of which is roughly
as follows. In the dense graph setting, the popular notions of graph limits
are all equivalent to left convergence, which says that a growing sequence of
graphs $G_j$ converges if, for each fixed graph $F$, the proportion of copies
of $F$ in $G_j$ converges. The first key insight is that this can be phrased
in probabilistic language by viewing left convergence as requiring
convergence in distribution of random subgraphs $H_{j,k}$ drawn by selecting
$k$ vertices independently from $G_j$, for all $k\in\Nats$. The second key
insight is that we may pass from nonrandom sequences of graphs
$(G_j)_{j\in\Nats}$ to sequences of random adjacency matrices
$(A(G_j))_{j\in\Nats}$ by randomly labeling the vertices of each $G_j$ by
$\{1, \dots, \vertices(G_j)\}$; this gives a construction such that for each
fixed $j$ the random adjacency matrix is exchangeable. We then observe that
convergence in distribution of randomly sampled $k$ vertex subgraphs is
equivalent to convergence in distribution of the random adjacency matrices
given by restricting $A(G_j)$ to its upper left $k\times k$ submatrix. Now,
using standard probability theory machinery, distributional convergence of
all size $k$ prefixes is enough for even distributional convergence of
$A(G_j)$ as $j \to\infty$. As one might expect, the limit of $A(G_1), A(G_2),
\dots$ is an infinite exchangeable array. By the Aldous--Hoover theorem,
there is then some graphon $W$ that characterizes the distribution of this
array. This graphon is the same as the left convergent limit of the graph
sequence $G_1, G_2, \dots$.

In the present context, the relationship between nonrandom graph sequences
and sequences of randomly labeled objects is captured as a correspondence
between edge sets and point processes. The point processes will be given in
terms of \emph{adjacency measures}, defined as locally finite measures of the
form $\xi= \sum_{i,j}\delta_{(\theta_i,\theta_j)}$, where the sum goes over
all ordered pairs $i,j$ such that $\{i,j\}$ is an edge of a countable graph
$G$ (possibly containing some loops, that is, edges joining a vertex to
itself) and $\theta_i\in\Reals_{+}$ with $\theta_i \neq\theta_j$ for $i\neq
j$.

\begin{defn}
Let $G$ be a labeled or unlabeled graph and let $s > 0$. A~\defnphrase{random
labeling of $G$ into $[0,s)$} is a random adjacency measure obtained by
labeling the vertices randomly with i.i.d. labels in $[0,s)$.
\end{defn}

For a graph sequence $G_1, G_2, \dots$ it may not be immediately obvious what
the ranges $[0,s_1),[0,s_2),\dots$ of the random labelings should be. Our
choice here is $s_j = \sqrt{2\edges(G_j)}$, which has the virtue that for all
bounded sets $A,B\subseteq\Reals_{+}$ such that $\max(A \cup B) \le s_j$, the
expected number of edges between vertices with labels in $A$ and $B$ is
independent of the graph.

\begin{defn}\label{def:canon-label}
We define the \defnphrase{canonical labeling} $\mathsf{Lbl}(G)$ of a graph
$G$ to be the random labeling of $G$ into $[0,\sqrt{2\edges(G)})$.
\end{defn}

The relationship between sampling convergence of a graph sequence and the
distributional convergence of the canonical labelings is closely analogous to
the relationship between left convergence of a graph sequence and the
distributional convergence of the associated random adjacency matrices. We
show in Section~\ref{sec:samp_lims} that the graph sequence $G_1, G_2, \dots
$ is sampling convergent to $\W$ if and only if the canonical labelings
$\mathsf{Lbl}(G_1), \mathsf{Lbl}(G_2), \dots$ converge in distribution to an
infinite exchangeable point process characterized by $\W$. Indeed, the
machinery of distributional convergence of point processes is core to many of
our main results.

In \cite{Aldous:2009}, a broad program for studying the limits of complex
structures of increasing size is outlined. The basic idea is to define a
notion of sampling on these structures such that for each complex object
$C_j$ we may sample some substructure $D^{(k)}_j$ of size $k$; convergence is
then defined as convergence in distribution of $D^{(k)}_j$ as $j\to \infty$
for all sizes $k$. The natural limit is then the joint distribution of the
limiting object for all sizes $k$. This object will have some symmetries
imposed by the sampling scheme, and so might admit some more compact
representation, which would then be the natural limit object. One perspective
on the present paper is that we realize this program for $p$-samplings of
families of growing graphs.

\subsubsection{Metric convergence}

One of the important tools in the theory of dense graph limits is the cut
distance between two graphs or graphons
\cite{Borgs:Chayes:Lovasz:Sos:Vesztergombi:2006:1}. The cut metric defines a
notion of distance that, essentially, captures how similar two graphs or
graphons look at low resolutions; see Figure~\ref{fig:dilated_emp_graphon}
below. We define cut distance formally in Section~\ref{prelim}. One of the
contributions of \cite{Borgs:Chayes:Cohn:Holden:2016} was to generalize the
cut distance to graphons supported on general $\sigma$-finite spaces, and in
particular for graphons $W\colon\Reals_{+}^2 \to[0,1]$, and to use this
notion to compare two graphs via an embedding of the space of graphs into the
space of graphons $W\colon\Reals_{+}^2 \to[0,1]$, mapping a graph $G$ into
what they called the stretched canonical graphon $W^{G,s}$ of $G$. Using this
embedding, \cite{Borgs:Chayes:Cohn:Holden:2016} then introduced the
``stretched cut distance'' between two graphs as the cut distance between the
stretched canonical graphons of these graphs. That paper developed a theory
of graph limits based on convergence in this stretched cut distance, where
the essential idea is to transform a sequence of graphs into a sequence of
stretched canonical graphons and ask for cut metric convergence of this
sequence; see Figure~\ref{fig:dilated_emp_graphon}. This turns out to
generalize the dense graph cut metric convergence, and the generalized limit
objects are the same generalized graphons that arise as limits in sampling
convergence.

\begin{figure}
\includegraphics[width=0.95\linewidth]{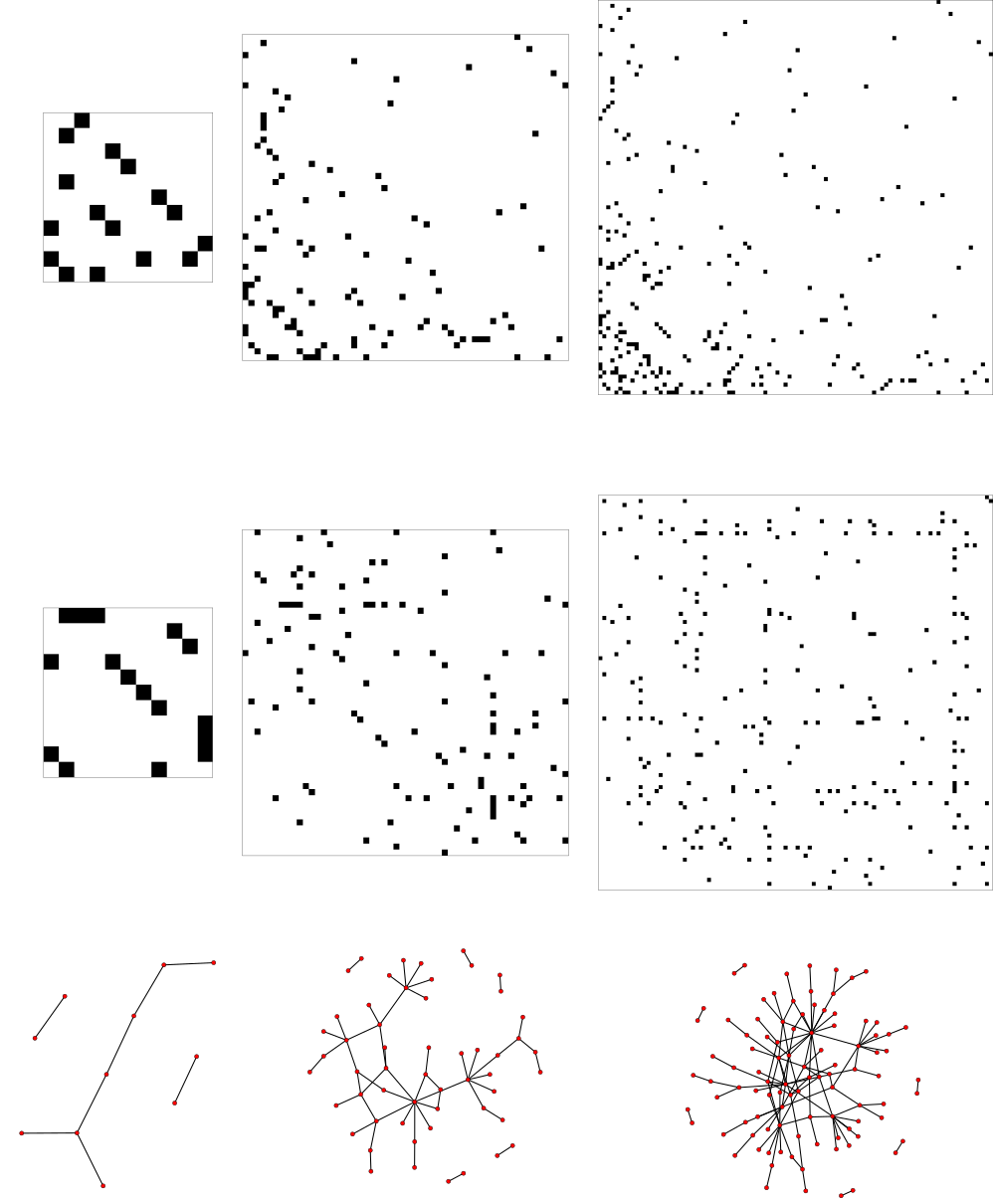}
\caption{Each column shows a graph (bottom row), a corresponding stretched empirical
graphon based on a random labeling of the vertices (middle row), and a
corresponding stretched empirical based on an alternative labeling (top row).
The three graphs are a prefix of a sequence that converges to $(0,0,W)$,
where $W(x,y)=(x+1)^{-2}(y+1)^{-2}$. Intuitively, the top row shows
pixel-picture approximations to the limiting graphon. The cut metric formalizes
this intuition: the graphons are aligned according to some optimal measure
preserving transformation, and the distance between them is then $\sup_{U,V
\subseteq\Reals_{+}}  \lvert\int_{U \times V} W_1(x,y) -
W_2(x,y) \,\intd x \,\intd y  \rvert$, the largest difference in
any patch between the total amounts of
ink in that
patch.}
\label{fig:dilated_emp_graphon}
\end{figure}

In the dense graph setting, convergence in cut distance is equivalent to left
convergence. Given that sampling convergence is an analogue of left
convergence, it is natural to expect that there should be some connection
with convergence under the stretched cut distance. Indeed this is so, and in
Theorem~\ref{cut_implies_sampling} we show that the two notions of
convergence coincide for any graph sequence that is subsequentially
convergent with respect to the stretched cut metric. Thus, in particular,
convergence under the stretched cut distance implies sampling convergence.

Our main motivation for the introduction of sampling convergence is
conceptual clarity. However, it is also worth noting that sampling
convergence (and the associated move from graphons to graphexes) has some
pleasant mathematical properties that stretched cut convergence does not. In
particular, every graph sequence is subsequentially sampling convergent, but
this is not true for stretched cut metric convergence.

\subsubsection{Consistent estimation}\label{GS_conv}

The paper \cite{Veitch:Roy:2016} deals with the problem of estimating $\W$
from a growing sequence of unlabeled graphs $G_1, G_2, \dots$ generated from
$\W$. Simplifying somewhat, the data set is modeled as $G_j = \mathcal{G}
(\KEG_{s_j} ) $ for some sequence $s_1, s_2, \dots$ of observation {time}s
with $s_j \upto\infty$ and $(\KEG_s)_{s\in\Reals_{+}}$ generated by $\W$. The
basic goal of estimation is to produce a sequence of graphexes $\W_{G_1},
\W_{G_2}, \dots$ such that $\W_{G_j} \to\W$ as $j \to \infty$, for some
notion of convergence that formalizes the idea that the distribution defined
by the estimated graphex should be asymptotically the same as the
distribution defined by the true underlying graphex. In the graphex setting,
there are two natural distinct notions of estimation depending on whether the
observation {time}s are included as part of the observation; both of these
are closely related to the sampling convergence of the present paper.

Let $\mathrm{GPD}  (\W,s ) = \Pr(\mathcal{G}  (\KEG _s ) \in\cdot\given\W)$
denote the probability distribution over unlabeled {time} $s$ graphs
generated by $\W$, where GPD stands for graphex process distribution. In the
setting where the {time}s are known, estimation is formalized by defining
$\W_j \to_{\mathrm{GP}}\W$ as $j \to\infty$ to mean $\mathrm{GPD} (\W_j,s )
\to\mathrm{GPD}  (\W,s ) $ weakly as $j \to\infty$, for all $s\in\Reals_{+}$.
That is, $\W_1, \W_2, \dots $ estimates $\W$ if the random graphs generated
by the estimators converge in distribution to the random graphs generated by
$\W$.

For a graph $G$, define $\widehat{W}_{(G,s)}\colon[0,\vertices (G)/s]^2 \to
\{0,1\}$, the dilated empirical graphon of $G$ with dilation $s$, to be the
function given by representing the adjacency matrix\footnote {Implicitly,
this notion requires us to order the vertices of $G$, since otherwise it is
not clear which interval of length $1/s$ should be mapped to a given vertex;
we will choose an arbitrary, fixed ordering for each unlabeled, finite graph
$G$. All our subsequent notions do not depend on the particular ordering, and
hence are well-defined for unlabeled graphs, as well as graphs with vertices
labeled by labels in an unordered set.} of $G$ as a step function where each
pixel has size $1/s \times1/s$; see Figure~\ref{fig:dilated_emp_graphon}. The
estimator used by \cite{Veitch:Roy:2016} in the setting where the {time}s
$s_j$ are included as part of the observation are dilated empirical graphons
of $G_j$ with dilation $s_j$. The basic structure of estimation---map a
sequence of graphs to a sequence of graphons and define a notion of
convergence on the graphons---looks very similar to the development of
(stretched) cut metric convergence, and as with stretched cut convergence,
there is a close connection to sampling convergence: $\widehat{W}_{(G_j,s_j)}
\to_{\mathrm{GP}}\W$ is equivalent to $\mathsf{Smpl} ( {G_j}, \frac{r}{s_j})
\convDist \mathcal{G}  (\KEG_r ) $ for all $r \in \Reals_{+}$. To explain
this connection, we recall a pair of ideas from \cite{Veitch:Roy:2016}
(themselves adapted from \cite{Kallenberg:1999}). First, generating a sample
from $\mathrm{GPD}  (\widehat {W}_{(G_j,s_j)},r ) $ is equivalent to sampling
a subgraph from $G_j$ by selecting $\poiDist (\frac{r}{s_j} \vertices(G_j) )
$ vertices with replacement, and returning the vertex induced subgraph
without its isolated vertices. Second, this with-replacement sampling scheme
is asymptotically equivalent to $r/s_j$-sampling (without replacement). The
equivalence of the two notions of convergence follows immediately.

If $s_1, s_2, \dots$ are not included as part of the observation, then we
require a different approach to estimation. For graphexes of the form $\W=
(0,0,W)$, \cite{Borgs:Chayes:Cohn:Holden:2016} proves that $e(G_j)/s_j^2 \to
\frac{1}{2}\|W\|_1 \as$ as $j\to\infty$, and it is not hard to extend this
result to general integrable graphexes, showing that $e(G_j)/s_j^2 \to
\frac{1}{2}\|\W\|_1 \as$ as $j\to\infty$, where we define the $L^1$ norm of a
graphex $\W= (I,S,W)$ as $\|\W\|_1=\|W\|_1+\frac{1}2\|S\|_1+\frac{1}2I$. This
suggests making a canonical choice of $\|W\|_1 = 1$ and defining the
stretched canonical graphon $W^{G,s}$ of a graph $G$ as the dilated empirical
graphon of $G$ with dilation $\sqrt{2\edges(G)}$. The salient fact, spelled
out in Lemma~\ref{w_wo_pp_equiv}, is that $G_1, G_2, \dots$ is sampling
convergent to $\W$ if and only if $W^{G_j,s} \to_{\mathrm{GP}}\W$ as $j \to
\infty$. In conjunction with our result that graph sequences generated by
$\W$ are sampling convergent to $\W$, this establishes that the stretched
canonical graphon is a consistent estimator for $\W$ if $\|\W\|_1=1$.

Veitch and Roy (2016) follow a different approach. In the case where the
sample {time}s are not included as part of the observation, the most general
observation is the sequence of all distinct (unlabeled) graph structures
taken on by $(\mathcal{G}  (\KEG_s ))_{s\in \Reals_{+}}$; call this
collection $\GS{\KEG}$, the graph sequence of $\KEG$. Intuitively, this is
the structure that remains when the labels are stripped from $(\KEG_s)_{s\in
\Reals_{+}}$. The natural notion of estimation for graph sequences is then to
say that $\W_j \to_{\mathrm{GS}}\W$ as $j \to\infty$ whenever $\GS{\KEG^j}
\convDist \GS{\KEG}$, where $\KEG^j$ is generated by $\W_j$; that is, $\W_1,
\W_2, \dots $ estimates $\W$ if the distribution over unlabeled structures
generated by $\W _j$ is asymptotically equal to the distribution over
unlabeled structures generated by $\W$. It turns out that the empirical
graphon (without any dilation) is a consistent estimator for $\W$ in the
graph sequence sense; so indeed estimation is possible without any knowledge
of $s_1, s_2, \dots$.

Because the empirical graphon relies only on the graph (and not the latent
observation {time}), it can be used to define a notion of graph limit. Let
$G_1, G_2, \dots$ be a sequence of graphs (not necessarily corresponding to a
graphex process), and say that the sequence is \emph{GS convergent} to $\W$,
written $G_j \to_{\mathrm{GS}}\W$ as $j \to\infty$, whenever $W^{G_j}
\to_{\mathrm{GS}}\W$ as $j \to\infty$. \cite{Veitch:Roy:2016}, Lemma~5.6,
shows that as long as $\W\neq 0$, $\W_j \to_{\mathrm{GP}}\W$ as $j \to\infty$
implies also $\W_j \to _{\mathrm{GS}}\W$ as $j \to \infty$, from which it
follows that sampling convergence implies GS convergence. The converse is not
true: the consistent estimation results of \cite{Veitch:Roy:2016} establish
that graph sequences generated by nonintegrable $\W$ are GS convergent to
$\W$, but sampling convergent limits are always integrable. Thus GS
convergence provides an even more general notion of graph limit. However, it
is unclear whether GS convergence has any interpretation or motivation
outside the graphex process theory.

We note that \cite{Janson:2017} includes a discussion of the relationship
between various notions of convergence of graphexes, and is closely related
to the development in this section.

\subsection{Organization}

We give formal definitions and recall some important results in
Section~\ref{prelim}. The basic results for sampling convergence---most
importantly, the limits are graphexes---are given in
Section~\ref{sec:samp_lims}. In Section~\ref{sec:gp_samp_conv}, we prove that
a graph sequence generated by integrable graphex $\W$ is almost surely
sampling convergent to a canonical dilation of $\W$; this has the particular
consequence that (a canonical representative of) every integrable graphex
arises as the sampling limit of some graph sequence. In
Section~\ref{sec:met_and_samp_dist}, we relate convergence in distribution of
graphex sequences generated by $\W_1, \W_2, \dots$ to the metric convergence
of \cite{Borgs:Chayes:Cohn:Holden:2016}. In Section~\ref{sec:metrization}, we
metrize sampling convergence and show that the metric completion of the space
of finite unlabeled loopless graphs is compact (a less elegant statement is
required for loops). In Section~\ref{sec:samp_defines_GP}, we prove that if a
graph-valued stochastic process $(G_s)_{s\in\Reals_{+}}$ has the property
that, for all $p \in(0,1)$ and all $s \in\Reals_{+}$, a~$p$-sampling of $G_s$
is equal in distribution to $G_{p s}$, then there is some graphex $\W$ such
that $G_s = \mathcal{G}
 (\KEG_s ) $
for some $(\KEG_s)_{s\in\Reals_{+}}$ generated by $\W$.

\section{Preliminaries}\label{prelim}

As usual, we denote the set of edges and vertices of a graph $G$ by $E(G)$
and $V(G)$, respectively. In general, $E(G)$ will consist of both loop and
nonloop edges; we denote the number of nonloop edges by $e(G)$ and the number
of loop edges by $\loops(G)$.

Some of the basic objects of interest in this paper are locally finite point
processes on $\Reals_{+}^2$, interpreted as the edge sets of random graphs
with vertices labeled in $\Reals_{+}$. Here, as usual, a \defnphrase {locally
finite point process} on $\Reals_{+}^2$ is a random element $\xi$ of the set
$\mathcal N=\mathcal N(\Reals_{+}^2)$ of locally finite counting measure on
$\Reals_{+}^2$ (i.e., the set of integer valued measures $\xi$ such that $\xi
(A)<\infty$ for all bounded Borel sets $A\subset\Reals_{+}^2$), equipped with
the Borel $\sigma$-algebra inherited from the vague topology, defined as the
coarsest topology for which the maps $\mu\mapsto\int f \,\intd\mu$ are
continuous for all continuous functions with bounded support. As shown in,
for example, \cite{Daley:Vere-Jones:2003:v2}, this topology can be metrized
in such a way that $\mathcal N$ becomes a complete, separable metric space.
Convergence in distribution for locally finite point processes is defined as
weak convergence with respect to this topology, so that $\xi_n\convDist\xi $
is defined by the condition that $\EE[ F(\xi_n)]\to\EE[F(\xi)]$ for all
continuous, bounded functions $F$, with continuity defined with respect to
the vague topology on $\mathcal N$.

\begin{defn}
An \defnphrase{adjacency measure} is a purely atomic, symmetric locally
finite counting measure on $\Reals_{+}^2$ for which all atoms have weight
$1$. A \defnphrase{random adjacency measure} is a locally finite point
process $\xi$ on $\Reals_{+}^2$ such that $\xi$ is almost surely an adjacency
measure.
\end{defn}

We associated a graph with labels in $\Reals_{+}$ to an adjacency measure
$\xi$ by writing it as $\xi= \sum_{i,j}\delta_{(\theta_i,\theta_j)}$,
defining the set $\{(\theta_i,\theta_j)\}$ with $\theta_i \le\theta_j$ as its
edge set, and defining the set of points $\theta_i$ that participate in at
least one edge as its vertex set. Most of the time, we will not distinguish
between the countable graph associated with $\xi$ and the adjacency measure
$\xi$ itself.

The defining property of graphex processes is that, intuitively speaking, the
labels of the vertices of the graphs are uninformative about their structure.
This is formalized by requiring the associated adjacency measure to be
jointly exchangeable.

\begin{defn}
A random adjacency measure $\xi$ is \defnphrase{jointly exchangeable} if $\xi
\circ(\phi\otimes\phi) \equaldist\xi$ for every measure-preserving
transformation $\phi\colon\Reals_{+}\to\Reals_{+}$. It is called an
\defnphrase{extremal exchangeable adjacency measure} if its distribution
cannot be written as a nontrivial superposition of distributions over jointly
exchangeable adjacency measures, that is, if a representation of the
distribution as $\alpha P_1 + (1-\alpha) P_2$ for some $\alpha\in(0,1)$
implies that $P_1=P_2$ a.e.
\end{defn}

A representation theorem for jointly exchangeable random measures on
$\Reals_{+}^2$ was given by Kallenberg
\cite{Kallenberg:2005,Kallenberg_Random_Meas_Plane}. This result was
translated to the setting of random graphs in
\cite{Veitch:Roy:2015,Borgs:Chayes:Cohn:Holden:2016}. Writing $\Lebesgue$ for
Lebesgue measure and $\mu_W(\cdot) = \int _{\Reals_{+}} W(x,\cdot) \,\intd
x$, the defining object of the representation theorem is as follows.

\begin{defn}
\label{def:graphex} A \defnphrase{graphex} is a triple $(\IsoF ,\StarF,W)$,
where $\IsoF\ge0$ is a non-negative real, $\StarF\colon\Reals_{+} \to
\Reals_{+}$ is a measurable function such that $\min(S,1)$ is integrable, and
the \defnphrase{graphon} $W \colon\Reals_{+}^2 \to[0,1]$ is a symmetric,
measurable function that satisfies:
\begin{enumerate}
\item[1.]$\Lebesgue\{ \mu_W = \infty\} = 0$ and $\Lebesgue\{\mu_W > 1\} <
    \infty$,
\item[2.]$\int_{\Reals_{+}^2} W(x,y) 1[ \mu_W(x) \le1 ]  1[\mu_W(y) \le1]
    \,\intd x \,\intd y < \infty$, and
\item[3.]$\int_{\Reals_{+}} W(x,x) \,\intd x < \infty$.
\end{enumerate}
\end{defn}

\begin{remark}
Integrability of $W$ (and its diagonal) is a sufficient but not necessary
condition for it to be a graphon. If the graphon, its diagonal
$W_D(x)=W(x,x)$, and the function $S$ are integrable, then we say that $\W$
is an \defnphrase{integrable graphex}. We set
\[
\llVert \W \rrVert _1= \llVert W \rrVert _1+2 \llVert S
\rrVert _1+ 2 I.
\]
Integrability plays a fundamental role in sampling convergence.
\end{remark}

Each graphex gives rise to a random adjacency measure, which in turn leads to
a graph-valued stochastic process:

\begin{defn}
Given a graphex $\W=(\IsoF,\StarF,W)$, let $\xi$ be the random adjacency
measure
%
\begin{align}
\label{KEGgen} \xi =& \sum_{i,j}1\bigl[
\zeta_{\{i,j\}} \le W(\vartheta_{i},\vartheta _{j})
\bigr]\delta_{\theta_{i},\theta_{j}}
\nonumber
\\
& {} + \sum_{j,k} 1\bigl[ \chi_{jk} \le
\StarF(\vartheta _{j}) \bigr](\delta_{\theta_{j},\sigma_{jk}}+
\delta_{\sigma_{jk},\theta
_{j}})
\nonumber
\\
& {} + \sum_{k} 1[\eta_{k} \le\IsoF](
\delta_{\rho
_{k},\rho'_{k}}+\delta_{\rho'_{k},\rho_{k}}), %
\end{align}
where $(\zeta_{\{i,j\}})$ is a collection of independent uniformly
distributed random variables in $[0,1]$, $\{ (\theta_{j},\vartheta _{j})\} $
and $\{ (\sigma_{ij},\chi_{ij})\} _{j}$, for $i\in\Nats$, are independent
unit rate Poisson processes on $\Reals_{+}^{2}$, and $\{
(\rho_{j},\rho_{j}^{\prime},\eta_{j})\} $ are independent unit rate Poisson
processes on $\Reals_{+}^{3}$, all of them independent of each
other.\footnote{By the results of \cite {Kallenberg_Random_Meas_Plane}, the
integrability conditions from Definition~\ref{def:graphex} imply that the
above sums are a.s. convergent in the vague topology, which in turn implies
that $\xi$ is a.s. locally finite. It is furthermore not hard to show that
a.s., $\xi$ is simple, implying that $\xi$ is an adjacency measure.}

Let $\KEG$ be the (in general countably infinite) graph corresponding to the
adjacency measure $\xi$ defined in (\ref{KEGgen}), and let $\KEG_s$ be the
(a.s. finite) graph corresponding to the adjacency measure $\xi _s(\cdot) =
\xi(\cdot  \cap[0,s]^2)$. The
\defnphrase{graphex process}
associated with graphex $(\IsoF,\StarF,W)$ is the family
$(\KEG_s)_{s\in\Reals_{+}}$. \label{def:graphex-process}
\end{defn}

\begin{remark}
\label{rem:graphex-4} One might be tempted to identify graphexes that are
equal almost everywhere. While this is possible, one must pay attention to
details here, since changing a graphon on the diagonal is only a change on a
set of measure zero, but it changes the graphex process associated to the
graphex. This problem can be easily addressed by introducing the function
$W_D(x)=W(x,x)$, identifying a graphex $(I,S,W)$ with the quadruple
$(I,S,W,W_D)$ and considering the latter as an element of $\Reals_{+} \times
L^0(\Reals_{+},\Lambda)\times L^0(\Reals_{+}^2,\Lambda^2)\times
L^0(\Reals_{+},\Lambda)$.
\end{remark}

\begin{remark}
In \cite{Borgs:Chayes:Cohn:Holden:2016}, a nominally more general definition
of a graphon (and the associated graphon process) is used. There the domain
of $W$ is allowed to be an arbitrary $\sigma$-finite measure space modeling a
space of latent features. The associated process is then defined by labeling
vertices with a pair of labels, namely their birth {time} and their feature.
In the above definition of $(\KEG_s)_{s\in\Reals_{+}}$, the feature space is
assumed to be $\Reals_{+}$, and vertices are just labeled by their birth
{time}s, not a pair of labels. By Theorem~\ref{theorem:graphex_rep_theorem}
below, or the explicit measure-preserving mappings constructed in
\cite{Borgs:Chayes:Cohn:Holden:2016}, every such model is equivalent to one
with latent feature space $\Reals_{+}$, so there is no loss of generality in
our definition. The motivation for the more general notion is that in many
situations there is a natural choice for the space of latent features, and
strong-arming the feature space to $\Reals_{+}$ may obfuscate the conceptual
underpinnings of the model or destroy certain nice theoretical properties
(such as continuity of the graphon). In the present paper, this is not a
concern, so we prefer the simpler definition with graphons defined over
$\Reals_{+}$. We also label vertices in the graphex process
$(\KEG_s)_{s\in\Reals_{+}}$ by just their birth {time}, since in this paper,
the latent feature of a vertex is usually not important. Indeed, as we will
see below, we often remove even the birth {time} label of our vertices,
leading to processes of unlabeled graphs.
\end{remark}

Given Definitions \ref{def:graphex} and \ref{def:graphex-process}, we can now
state the Kallenberg representation theorem.

\begin{theorem}\label{theorem:graphex_rep_theorem}
Let $\xi$ be a random adjacency measure. Then $\xi$ is jointly exchangeable
iff there exists a (possibly random) graphex $\W$ such that $\xi$ is of the
form \eqref{KEGgen}. The graphex $\W$ can be chosen to be nonrandom if and
only if $\xi$ is extremal.
\end{theorem}

\begin{remark}
By a random measurable function $f\colon X \to\Reals$, we mean a measurable
function $f'\colon[0,1] \times X\to\Reals$ and a randomization variable
$\alpha \sim \uniDist[0,1]$ such that $f(x) = f'(\alpha,x)$; see, for
example, \cite{Gikhman:Skorokhod:1969}, Chapter~4. By a random graphex, we
mean a quadruple $(I,S,W,W_D)$ (see Remark~\ref{rem:graphex-4} above) such
that each component is an appropriate random measurable function all sharing
a common randomization variable $\alpha$, and such that the graphex
integrability requirements are almost surely satisfied; by conditioning on a
graphex $\W$ we mean conditioning on the randomization parameter $\alpha$. We
separate out the diagonal of the graphon so that two graphexes that are equal
a.e. generate the same distribution over adjacency measures; this sidesteps
some measurability technicalities.
\end{remark}

We will often have occasion to refer to the unlabeled finite graph associated
with a finite adjacency measure.

\begin{defn}
Let $\xi$ be a finite adjacency measure. The \defnphrase{unlabeled graph
associated with $\xi$} is $\mathcal{G}  (\xi ) $.
\end{defn}

Similarly, we will often want to move from unlabeled graphs to adjacency
measures. To do so, we must invent labels for the vertices; a simple scheme
is to produce labels independently and uniformly in some range:

\begin{defn}
Let $G$ be a graph with edge set $E$, and let $s > 0$. A
\defnphrase{random labeling of $G$ into $[0,s)$}, denoted
$\mathsf{Lbl}_{s}(G,\{U_i\})$, is a random adjacency measure
$\mathsf{Lbl}_{s}(G,\{U_i\}) = \sum_{(i,j) \in E} \delta_{(U_i,U_j)}$, where
the sum contains both orientations of each nonloop edge and $U_i
\distiid\uniDist[0,s)$ for each vertex $i$ in $G$. Where there is no risk of
confusion, we will write $\mathsf{Lbl}_{s}(G)$ for
$\mathsf{Lbl}_{s}(G,\{U_i\})$ where $U_i \distiid \uniDist[0,s)$ for all
vertices $i$, independently of everything else. The random labeling is called
a
\defnphrase{canonical labeling of $G$} and denoted by $\mathsf
{Lbl}(G)$ if $s=\sqrt{2e(G)}$.
\end{defn}

Note that for an unlabeled graph, we need to fix a labeling of the vertices
of $G$ to define $\mathsf{Lbl}_{s}(G,\{U_i\})$; since the distribution of
$\mathsf{Lbl}_{s}(G,\{U_i\})$ is independent of what labeling we chosen for
$G$, the choice of this labeling is irrelevant.

\subsection{Graph limits}

We now recall some important definitions and results on the metric
convergence of \cite{Borgs:Chayes:Cohn:Holden:2016}, specializing to the case
of graphons defined over $\Reals_{+}^2$ and sequences of simple graphs.

There are two main notions of distance between integrable graphons that we
will need. The first is a modification of the $L^1$ distance that accounts
for the fact that graphons have a natural equivalence under measure
preserving transformations. For $\psi\colon\Reals_{+}\to\Reals_{+}$, we let
$W^\psi(x,y) = W(\psi(x),\psi(y))$.

\begin{defn}
The \defnphrase{invariant $L^1$ distance} between integrable graphons $W_1,
W_2$ is $\delta_1(W_1,W_2) = \inf_{\psi_1,\psi_2}\|W_1^{\psi_1} -
W_2^{\psi_1}\|_1$, where the infimum is over all measure-preserving
transformations $\psi_j\colon\Reals_{+}\to\Reals_{+}$ for $j=1,2$.
\end{defn}

Intuitively, the invariant $L^1$ distance lines up the two graphons as
closely as possible and then takes the $L^1$ distance between them.

The invariant $L^1$ distance is too stringent of a notion for many cases of
interest. In particular, it is obviously impossible to approximate a general
graphon by a $\{0,1\}$-valued graphon under that notion of distance. The
weakened distance we use is as follows.

\begin{defn}
The \defnphrase{cut distance} between two integrable graphons $W_1, W_2$ is
\[
\delta_\square(W_1,W_2) = \inf
_{\psi_1,\psi_2} \sup_{U,V
\subseteq\Reals_{+}} \biggl\lvert
\int_{U \times V} W_1^{\psi_1}(x,y) -
W_2^{\psi_2}(x,y) \,\intd x \,\intd y \biggr\rvert,
\]
where the infimum is over all measure-preserving transformations
$\psi_j\colon\Reals_{+}\to\Reals_{+}$ for $j=1,2$ and the supremum is over
Borel sets $U,V \subseteq\Reals_{+}$.
\end{defn}

Intuitively, the cut distance lines up two graphons as closely as possible,
then ``smears them out'' so that they are close in the cut sense if their
mass on every rectangular region is close. This allows a $\{0,1\}$-valued
graphon to approximate an arbitrary graphon as a pixel-picture approximation
to a grayscale image; see Figure~\ref{fig:dilated_emp_graphon}.

The cut metric defines a form of convergence for sequences of integrable
graphons. To lift this to convergence of sequences of graphs, we need a
canonical way to map graphs to graphons.

\begin{defn}
The \defnphrase{empirical graphon} $W^{G}\colon[0,1]^2 \to \{0,1\}$ of a
graph $G$ is the function produced by partitioning $[0,1]^2$ into a
$\vertices(G) \times\vertices(G)$ grid and setting square $(i,j)$ to take
value $1$ if edge $(i,j)$ is included in $G$, and $0$ otherwise.
\end{defn}

The empirical graphon is the ``right'' mapping in the dense graph setting,
but it needs a modification in the sparse graph setting.

\begin{defn}
The \defnphrase{stretched canonical graphon} $W^{G,s}\colon \Reals_{+}^2
\to\{0,1\}$ of a graph $G$ is defined to be
\[
W^{G,s}(x,y) = W^{G} \bigl( \bigl\llVert W^G
\bigr\rrVert _1^{1/2} x, \bigl\llVert W^G \bigr
\rrVert _1^{1/2} y \bigr)
\]
if $x,y\in [0,\|W^G\|_1^{-1/2} )$ and $W^{G,s}(x,y)=0$ otherwise.
\end{defn}

See Figure~\ref{fig:dilated_emp_graphon}. The basic intuition for this
definition is that $\|W^{G,s}\|_1 = 1$, so that if $H_r  \sim \mathrm{GPD}
(W^{G,s},r ) $ then $\EE[\edges(H_r)] = r^2/2$. That is, the canonical
stretched graphon is stretched such that the corresponding graphon process
has a fixed ``growth rate'' irrespective of the graph used as input.

We now have an obvious notion for convergence of graph sequences.

\begin{defn}
A graph sequence $G_1, G_2, \dots$ \defnphrase{converges in stretched cut
distance} to $W$ if $\delta_\square(W^{G_j,s}, W) \to0$ as $j \to \infty$.
\end{defn}

A key property of stretched cut convergence is, by
\cite{Borgs:Chayes:Cohn:Holden:2016}, Theorem~28, if $(G_s)_{s\in\Reals_{+}}$
is a graphon process generated by $W$ such that $\|W\|_1 = 1$ then, almost
surely, $\delta_\square(W^{G_s,s}, W) \to0$ as $s \to\infty$. In this paper,
we will establish the analogous result for sampling convergence.

The space of graphons equipped with the cut metric is not relatively compact,
so a further restriction is needed for subsequential convergence.

\begin{defn}\label{def:tail-reg}
A set of graphons $\{W_j\}_{j\in\Nats}$ has \defnphrase{uniformly regular
tails} if for every $\epsilon> 0$ there is some $M > 0$ such that for each
$j$ there is some $U_j \subseteq\Reals_{+}$ with $ \lvert U_j
 \rvert < M$ and
$\|W_j - W_j 1_{U_j \times U_j}\|_1 < \epsilon$ for all $j$. A~set of graphs
$\{G_j\}_{j\in\Nats}$ is said to have uniformly regular tails if
$\{W^{G_j,s}\}_{j\in\Nats}$ has uniformly regular tails.
\end{defn}

The main results about sequences with uniformly regular tails are that any
such sequence has a further subsequence that converges in cut distance---that
is, any such sequence is relatively compact in cut distance---and that any
sequence that is convergent in cut distance also has uniformly regular tails
(see \cite{Borgs:Chayes:Cohn:Holden:2016}, Corollary~17). Intuitively
speaking, the uniformly regular tail condition requires the graphs to have
``dense cores,'' where a constant fraction of all edges of $G_j$ occur
between only $\Theta(\sqrt{\edges(G_j)})$ vertices.

\subsection{Sampling}

Sampling convergence requires subgraphs sampled from $G_1, \break G_2, \dots$
to converge in distribution to finite size random graphs given by dropping
the labels from finite size graphex processes. It is most convenient to
express this by introducing notation for the distributions of these graphs.

\begin{defn}
The \defnphrase{canonical sampling distribution} with parameters $s$ and $G$
is $\mathrm{SmplD}  (G,s ) (\cdot)=\Pr (\mathsf {Smpl} (G,s /\sqrt
{2\edges(G)}) \in\cdot  |  G ) $.
\end{defn}

\begin{defn}
Let $(\KEG_s)_{s\in\Reals_{+}}$ be a graphex process generated by $\W $, with
$\W$ possibly random. The \defnphrase{unlabeled graphex process distribution}
with parameters $\W$ and $s$ is $\mathrm{GPD}  (\W,s )
(\cdot)=\Pr(\mathcal{G}  (\KEG _s ) \in \cdot\given\W)$.
\end{defn}

Instead of $\mathsf{Smpl}  (G_{j},s /\sqrt {2\edges(G_{j})})
\convDist\mathcal{G}  (\KEG_s ) $ as $j \to\infty$, we may now equivalently
write $\mathrm{SmplD} (G_j,s ) \to \mathrm{GPD}  (\W,s ) $ weakly as $j
\to\infty$. This has the advantages that it makes the limit object $\W$
explicit, it does not introduce extraneous randomness (nonrandom graphs are
mapped to nonrandom probability measures), and it allows us to deal easily
with cases where the graph sequence or $\W$ is random.

\begin{defn}\label{def:smpl-limit}
Let $\W$ be a graphex and let $G_1,G_2,\dots$ be a sequence of graphs. We say
that $G_1,G_2,\dots$ is \emph{sampling convergent} if $\mathrm {SmplD}
(G_j,s ) $ converges weakly as $j \to\infty$ for every $s$. We say that the
sequence is \emph{sampling convergent to $\W$} or \emph{sampling convergent
with limit $\W$} if $\mathrm{SmplD}  (G_j,s ) \to\mathrm{GPD}  (\W,s ) $
weakly as $j \to\infty$ for every $s$.
\end{defn}

We will make use of another sampling scheme that is asymptotically equivalent
to $p$-sampling with $p={r}/{\sqrt{2\edges(G_j)}}$. The alternative sampling
scheme will again be defined for labeled or unlabeled input graphs and, as in
the case of $p$-sampling, outputs an unlabeled graph, whether the input graph
is labeled or not.

\begin{defn}
A \defnphrase{with-replacement $p$-sampling} $\mathsf{SmplWR}(G,p)$ of a
graph $G$ is an unlabeled graph obtained by sampling $\poiDist(p
\vertices(G))$ vertices from $G$ with replacement and returning the
vertex-induced ``subgraph'' without its isolated vertices. Explicitly, if
$x_1,\dots,x_k$ are the vertices of $G$ chosen by sampling with replacement,
we first form a graph on $[k]$ by joining $i,j\in[k]$ by an edge whenever
$(x_i,x_j)$ is an edge in $G$ (whether that edge was a loop or an edge
between two different vertices), then deleting isolated vertices, and then
returning the resulting graph without its labels.
\end{defn}

The motivation for this definition is the observation that generating a
{time}-$r$ graph according to the canonical stretched empirical graphon of
$G$ is equivalent to a with-replacement ${r}/{\sqrt{2\edges(G)}}$-sampling of
$G$, in the sense that
\[
\mathrm{GPD} \bigl(W^{G,s},r \bigr) = \Pr \biggl(\mathsf{SmplWR}
\biggl({G},\frac{r}{\sqrt{2\edges(G)}}\biggr) \in\cdot | G \biggr).
\]
This observation (essentially) originates in \cite{Veitch:Roy:2016}, in the
context of the study of the empirical graphons of $G_1, G_2, \dots$ generated
by $\W$ at {time}s $s_1, s_2, \dots$, and stretched out by a factor of $s_j$
at each stage (instead of $\sqrt{2\edges(G_j)}$). In our setting, there is a
small additional complication arising from possible loops in $G$.

Recall that $\loops(G)$ denotes the number of loops of a graph $G$.
Asymptotic equivalence of with and without replacement sampling translates to
the following lemma.

\begin{lem}\label{lem:coupling}
Let $G$ be a random graph with $\edges$ edges and $\loops$ loops, and let
$p\leq1$. Then $\mathsf{Smpl} (G,p)$ and $\mathsf{SmplWR} (G,p)$ can be
coupled in such a way that a.s.,
\begin{align*}
\Pr\bigl(\mathsf{Smpl} (G,p)\neq\mathsf{SmplWR} (G,p) \given G\bigr) &\leq4
p^3 \edges+ 2 p^2 \loops.
\end{align*}
\end{lem}

\begin{proof}
Note that $\EE[\vertices(\mathsf{Smpl} (G,p)) \given G] \le2 p^2 \edges+ p
\loops$. \cite{Janson:2017}, Lemma~5.2, establishes that there exists a
coupling such that, almost surely,
\[
\Pr\bigl(\mathsf{Smpl} (G,p)\neq\mathsf{SmplWR} (G,p) \given\mathsf{Smpl} (G,p),
G\bigr) \le 2 p \vertices\bigl(\mathsf{Smpl} (G,p)\bigr).
\]
The result follows immediately.

We note that \cite{Janson:2017}, Lemma~5.2, does not explicitly treat graphs
with loops, but the proof given there applies verbatim to this case.
\end{proof}

\begin{lem}\label{samp_equiv_poi_wr}
Let $G_1, G_2, \dots$ be a sequence of (possibly random) graphs such that
a.s., $G_j$ is finite, $\edges(G_j) \to\infty$, and $\loops(G_j) =
O(\sqrt{\edges(G_j)})$ as $j\to\infty$. Then a.s. with respect to the
randomness of the sequence $G_1, G_2, \dots$, we have that $\mathsf{Smpl}
(G_j,{r}/{\sqrt{2\edges(G_j)}}) \convDist H$ for some finite random graph $H$
if and only if\break $\mathsf{SmplWR} (G_j,{r}/{\sqrt{2\edges(G_j)}}) \convDist H$.
\end{lem}

\begin{proof}
The proof follows immediately from the previous lemma by setting
$p=r/\sqrt{2\edges(G_j)}$, $\edges=\edges(G_j)$, and $\loops=\loops(G_j)$.
\end{proof}

\subsection{Coupling}

Much of this paper involves convergence of probability measures. We will
often make use of coupling techniques in order to establish these results;
see \cite{denHollander:2012} for an overview. A~coupling of probability
measures $P$ and $P'$, both on the measurable space $(E,\mcl{E})$, is a
probability measure $\widehat{P}$ on $(E\times E, \sigma(\mcl{E} \times
\mcl{E}))$ with marginals $P$ and $P'$. Such a coupling $\widehat{P}$ bounds
the total variation distance $ \llVert P-P' \rrVert _{\mathrm{TV}}$ between
$P$ and $P'$ by
\[
\bigl\llVert P-P' \bigr\rrVert _{\mathrm{TV}} \le\widehat{P}
\bigl(X \neq X'\bigr),
\]
where $X$ and $X'$ are random variables on $E$ with distributions $P$ and
$P'$ (which we then view as functions of the two coordinates on $E \times
E$). Moreover, if $E$ is a Polish space, then there exists some coupling that
saturates this bound.

It is often convenient to describe a coupling as a scheme for jointly
sampling $X$ and $X'$. In this case, we may refer to the coupling as a
coupling of the random variables. In this case, the basic proof technique is
to describe an algorithm for jointly sampling $X$ and $X'$, and then bound
$\Pr(X \neq X')$ under this algorithm.

\subsection{Distributional convergence of point processes}

Our technical development relies on techniques from point process theory,
particularly the theory of distributional convergence of point processes
viewed as random measures. Good references include
\cite{Daley:Vere-Jones:2003:v1,Daley:Vere-Jones:2003:v2} for a friendly
introduction and \cite{Kallenberg:2002}, Chapter~16, for a very general
treatment.

For our purposes, the main result needed to understand distributional
convergence of point processes is the following theorem.

\begin{theorem}[\cite{Daley:Vere-Jones:2003:v2}, Theorem~11.1.VII]
Let $\xi,\xi_1,\xi_2,\dots$ be locally finite point processes on
$\Reals_{+}^2$. Then $\xi_j \convDist\xi$ as $j \to\infty$ if and only if
\[
\bigl(\xi_j(B_1), \dots, \xi_j(B_n)
\bigr) \convDist\bigl(\xi(B_1), \dots, \xi(B_n)\bigr)
\]
as $j \to\infty$, where $B_i \subseteq\Reals_{+}^2$ are bounded Borel sets
such that $\Pr(\xi(\partial B_i) = 0) =1$.
\end{theorem}

That is, convergence in distribution of point processes is just convergence
in distribution of the counts on arbitrary collections of test sets. There
are generally consistency requirements between the counts on different test
sets, and in consequence it actually suffices to check convergence on a
smaller collection.

\section{Sampling limits of graph sequences}\label{sec:samp_lims}

In this section, we show that for graph sequences with size going to infinity
the limits of sampling convergence are graphexes.

The main technical idea is to use the canonical labeling to introduce a map
from graphs to probability distributions over point processes, and then
establish the claimed results by way of tools from the theory of
distributional convergence of point processes. Recall that the canonical
labeling of a graph $G$ is a random adjacency measure corresponding to
independently randomly labeling each vertex of $G$ uniformly in
$[0,\sqrt{2\edges(G)})$. We introduce notation for the probability
distribution of the random labeling.

\begin{defn}
The \defnphrase{embedding} of a (possibly random) graph $G$ is a probability
distribution over point processes on $[0,\sqrt{2\edges(G)})^2$ given by
\[
\mathrm{embed}(G) (\cdot) = \Pr\bigl(\mathsf{Lbl} (G) \in\cdot\given G\bigr).
\]
\end{defn}

Our first lemma relates distributional convergence of the point processes
given by the canonical random labelings of $G_1, G_2,\dots$ to sampling
convergence of the graph sequence. Intuitively, sampling convergence is
equivalent to distributional convergence of the point processes, and the
limiting random graph of $r/\sqrt{2\edges(G_j)}$-sampling is isomorphic to
the graph given by restricting the limiting adjacency measure to vertices
with label less than $r$. To parse the lemma statement, note that sampling
convergence may be written as, for all $r\in\Reals_{+}$, $\mathrm {SmplD}
(G_j,r ) $ converges weakly as $j \to\infty$. It may also be helpful to note
that part of our goal in this section is to establish that the limit
$\eta_{r}$ below is equal to $\mathrm{GPD}  (\W,r ) $ for some integrable
graphex $\W$.

\begin{lem}\label{cannon_emb_convs}
Let $G_1, G_2, \dots$ be a graph sequence with $\edges(G_j) \to \infty$ as
$j\to\infty$. The graph sequence is sampling convergent if and only if the
sequence\break
$\mathrm{embed}(G_1), \mathrm{embed}(G_2), \dots$ converges weakly,
that is, if and only if the random labelings converge in distribution.
Further, denoting the limiting distributions of $\mathrm{SmplD}  (G_j,r ) $
and $\mathrm{embed}(G_{j})$ by $\eta_{r}$ and $\labelLim$, respectively, if
$H_r  \sim \eta_{r}$ and $\xi \sim \labelLim$ then $\mathsf{Lbl}_{r}(H_r)
\equaldist\xi([0,r)^2 \cap \cdot)$.
\end{lem}

\begin{proof}
Suppose first that the sequence is sampling convergent. Fix $r$ and notice
that, for $\sqrt{2\edges(G_j)} > r$, under the canonical labelings of $G_j$
each vertex has a label in $[0,r)$ independently with probability
$r/\sqrt{2\edges(G_j)}$. Moreover, restricted to $[0,r)$, each vertex has a
$U[0,r)$ i.i.d. label. Denote this restriction by $\mathsf{Lbl}(G_j)| _r$. We
have just shown that $\mathsf{Lbl}(G_j)| _r \equaldist
\mathsf{Lbl}_r(\mathsf{Smpl}(G_j,r/{\sqrt{2\edges(G_j)}})$.
\cite{Veitch:Roy:2016}, Lemma~4.13, shows that if $G',G'_1,G'_2,\dots$ are
unlabeled random graphs then $G'_j \convDist G'$ as $j \to\infty$ if and only
if $\mathsf{Lbl}_r({G'_j}) \convDist \mathsf{Lbl}_r({G'})$ as $j \to\infty$.
Hence, by the assumption of sampling convergence, $\mathsf{Lbl}(G_j)| _r$
converges in distribution as $j \to \infty$.

Next, we lift this convergence on arbitrary prefixes $\mathsf {Lbl}(G_j)| _r$
to convergence of the entire point process. We first identify the limiting
point process $\xi$. To do so, we let $B_1, \dots, B_n \subseteq\Reals_{+}
^2$ be bounded Borel sets, choose $r$ such that $B_1, \dots, B_n
\subseteq[0,r)^2$, and demand that
\[
\bigl\{\xi(B_1), \dots, \xi(B_n)\bigr\} \equaldist\lim
_{j\to\infty}\bigl\{ \mathsf{Lbl}(G_j)|
_r(B_1), \dots, \mathsf {Lbl}(G_j)|
_r(B_n) \bigr\}.
\]
To see that the right-hand side is well-defined (i.e., independent of the
choice of $r$) notice that for $r<r'$, $(\mathsf{Lbl}(G_j)| _{r'})([0,r)^2
\cap \cdot) \equaldist\mathsf{Lbl}(G_j)| _r$. The right-hand side converges
in distribution because $\mathsf{Lbl}(G_j)| _r$ converges in distribution.
Moreover, the consistency conditions necessary for the right-hand side to be
counts with respect to some point process are satisfied, because the limiting
joint distributions are counts with respect to $\lim_{j\to\infty}
\mathsf{Lbl}(G_j)| _r$. By the Kolmogorov existence theorem for point
processes (see \cite{Daley:Vere-Jones:2003:v2}, Theorem~9.2.X), this suffices
to show that $\xi$ exists and has a well-defined distribution.

It is immediate that $\mathsf{Lbl}(G_j) \convDist\xi $ as $j \to\infty$
because, by construction,
\[
\bigl\{ \mathsf{Lbl}(G_j) (B_1), \dots, \mathsf
{Lbl}(G_j) (B_n) \bigr\} \convDist\bigl\{ \xi
(B_1), \dots, \xi(B_n) \bigr\} \qquad \text{as } j \to
\infty,
\]
for all bounded Borel sets $B_1, \dots, B_n \subseteq\Reals_{+}^2$.

The reverse direction follows similarly.
\end{proof}

The next result establishes that graphexes are the natural limit objects of
sampling convergent sequences.

\begin{lem}\label{lim_is_graphex}
Let $G_1,G_2, \dots$ be a sampling convergent graph sequence with
$\edges(G_j) \to\infty$ as $j\to\infty$. Then the limit is a graphex, in the
sense that there is some (possibly random) $\W$ such that if $\mathrm {SmplD}
(G_j,r ) \to Q_r$ then $Q_r \given\W= \mathrm{GPD}  (\W,r ) $.
\end{lem}

\begin{proof}
Notice that $\loops(G_j) = O(\sqrt{\edges(G_j)})$ for any sampling convergent
sequence, since otherwise the number of vertices in the random subgraph
diverges. By Lemma~\ref{cannon_emb_convs}, the canonical random labelings of
$G_j$ are convergent to some point process $\xi$ on $\Reals_{+}^2$. Observe
that for any $r$ and any measure-preserving transformation $\phi$ on $[0,r)$,
$\xi \circ(\phi\otimes\phi) \equaldist\xi$. In particular then, for any
dyadic partitioning of $\Reals_{+}$ and any transposition $\tau$ of this
dyadic partitioning, $\xi\circ(\tau\otimes\tau) \equaldist\xi$, and by
\cite{Kallenberg:2005}, Proposition~9.1, this implies that $\xi$ is
exchangeable. Then by the Kallenberg representation theorem there is some
(possibly random) graphex $\W$ that generates~$\xi$. That is,
$\mathrm{embed}(G_{j})$ converges weakly to the distribution over point
processes defined by (marginalizing over) $\W$. Lemma~\ref{cannon_emb_convs}
then establishes the result.
\end{proof}

We now turn to establishing that the limiting graphex $\W$ in
Lemma~\ref{lim_is_graphex} is nonrandom and integrable.

The next lemma gives a tractable criterion for determining when an
exchangeable point process is ergodic, that is, when $\W$ is nonrandom.
Basically, an adjacency measure is ergodic if for all $r,r'\in\Reals_{+} $
with $r<r'$, the induced subgraph with vertex labels less than $r$ gives no
information about the induced subgraph with vertex labels between $r$ and
$r'$. This lemma is an analogue of \cite{Kallenberg:2005}, Lemma~7.35,
attributed there to David Aldous.

\begin{lem}\label{ergodic_cond}
Let $\KEG$ be an exchangeable adjacency measure on $\Reals_{+}^2$. Then
$\KEG$ is extremal if and only if for all $r < r'\in\Reals_{+}$, $\KEG
([0,r)^2 \cap \cdot)$ and $\KEG([r,r')^2 \cap\cdot)$ are independent.
\end{lem}

\begin{proof}
If the point process is extremal, the Kallenberg representation theorem
(Theorem~\ref{theorem:graphex_rep_theorem}) immediately implies the result.

To prove the converse direction, we use the following notation from
\cite{Veitch:Roy:2016}. Let $\KEG$ be generated by $\W$ and let $\KEGDinf{\W}
= \Pr(\KEG\in\cdot\given\W)$, the (possibly random) probability measure over
adjacency measures induced by (the possibly random)~$\W$.

Suppose that $\KEG$ is not extremal. By a consistent estimation result
\cite{Veitch:Roy:2016}, Theorem~4.8, for any sequence $s_1, s_2, \dots$ such
that $s_j \upto\infty$,
\[
\lim_{j\to\infty}\Pr\bigl(\KEG\bigl([0,r)^2 \cap
\cdot) \in\cdot\given \KEG\bigl([r,s_j\bigr)^2 \cap
\cdot\bigr)\bigr) = \Pr\bigl(\KEG\bigl([0,r)^2 \cap\cdot\bigr) \in\cdot
\given \KEGDinf{\W}\bigr),
\]
(i.e., $\W$ can be estimated from an infinite size sample). Since, by
non-extremity, $\Pr(\KEG([0,r)^2 \cap\cdot) \in\cdot\given \KEGDinf{\W})
\neq\Pr(\KEG([0,r)^2 \cap\cdot) \in\cdot)$, this means that there is some $r'
\in\Reals_{+}$ such that
\[
\Pr\bigl(\KEG\bigl([0,r)^2 \cap\cdot) \in\cdot\given\KEG
\bigl([r,r'\bigr)^2 \cap \cdot\bigr)\bigr) \neq \Pr\bigl(\KEG
\bigl([0,r\bigr)^2 \cap\cdot) \in\cdot\bigr),
\]
as required.
\end{proof}

\begin{lem}\label{lim_is_deter}
The limiting graphex $\W$ in Lemma~\ref{lim_is_graphex} is nonrandom.
\end{lem}

\begin{proof}
As in the proof of Lemma~\ref{lim_is_graphex}, $\loops(G_j) =
O(\sqrt{\edges(G_j)})$ for any sampling convergent sequence. We then make use
of Lemma~\ref{samp_equiv_poi_wr}, the asymptotic equivalence of
$\mathsf{Smpl} (G_j,{r}/{\sqrt{2\edges(G_j)}})$ and $\mathsf{SmplWR}
(G_j,{r}/{\sqrt{2\edges(G_j)}})$. Let $r' \in\Reals_{+}$ and produce a
sequence of adjacency measures $\xi_{j,r'}$ by, for each $j \in \Nats$,
sampling a subgraph from $G_j$ according to the with replacement scheme (with
probability $r'/\sqrt{2\edges_j}$) and then randomly labeling this subgraph
in $[0,r')$. By the asymptotic equivalence of the sampling schemes and
Lemma~\ref{cannon_emb_convs}, $\xi_{j,r'} \convDist\xi([0,r')^2 \cap \cdot)$,
where $\xi$ is an adjacency measure generated by $\W$.

As a consequence of the with replacement sampling scheme, for all
$j\in\Nats$, $\xi_{j,r'}([0,r)^2 \cap\cdot)$ is independent of
$\xi_{j,r'}([r,r')^2 \cap\cdot)$, for any $r < r'$. To see this, note first
that each sampled vertex has a label in $[0,r)$ independently with
probability $r/r'$, so that, by a property of the Poisson distribution, the
number of vertices in $[0,r)$ and in $[r,r')$ have independent $\poiDist (r
\vertices(G_j)/\sqrt{2\edges(G_j)} ) $ and $\poiDist ((r'- r)
\vertices(G_j)/\sqrt{2\edges(G_j)} ) $ distributions. Second, because the
vertex sampling is with replacement, the structure of the graph with labels
in $[0,r)$ contains no information about the structure of the graph with
labels in $[r,r')$.

The independence of $\xi_{j,r'}([0,r)^2 \cap\cdot)$ and $\xi_{j,r'}([r,r')^2
\cap\cdot)$ for all $j \in\Nats$ implies that $\xi([0,r)^2 \cap \cdot)$ is
independent of $\xi([r,r')^2 \cap\cdot)$. Because $r,r'$ were arbitrary,
Lemma~\ref{ergodic_cond} implies that $\xi$ is ergodic, or, equivalently,
that $\W$ is nonrandom.
\end{proof}

Next, we show that the limiting $\W$ is integrable, we bound the integral and
we give a condition for when the bound is saturated. We will need the
following lemma.

\begin{lem}\label{expected_edges}
Let $(\KEG_s)_{s\in\Reals_{+}}$ be generated by $\W= (I, S, W)$. Then
\[
\EE\bigl[\edges(\KEG_s)\bigr] = \frac{s^2}2 \llVert \W
\rrVert _1 \quad \text{and}\quad \EE\bigl[\loops(\KEG_s)
\bigr]=s
\int W(x,x) \,\intd{x}.
\]
\end{lem}

\begin{proof}
Let $e^I_s$, $e^S_s$ and $e^W_s$ be the number of nonloop edges generated by
the $I$, $S$ and $W$ components of the graphex, noting that the edge sets
generated by the different components are disjoint.

The equation $\EE[e^I_s] = I s^2$ is immediate from Campbell's theorem.

By \cite{Veitch:Roy:2015}, Theorem~5.3, $\EE[e^W_s] = s^2 \frac{1}{2}\|W\|_1$
and $\EE[\loops(\KEG_s)] = s\int W(x,x) \,\intd{x}$.

To treat the star component, let $\PP_s$ be the latent Poisson process,
restricted to $[0,s) \times\Reals_{+}$, used to generate $\KEG$ and for each
$(t_i,x_i) \in\PP_s$ let $M(x_i)$ be the number of rays that $(t_i,x_i)$ has
due to the star component of the graphex. By viewing $M(x_i)$ as a marking of
$\PP_s$, and recalling that $M(x_i)  \sim \poiDist(sS(x_i))$, we have from
Campbell's theorem that $\EE[e^S_s] = s^2 \|S\|_1$.
\end{proof}

By construction,
\[
\EE \bigl[\edges \bigl(\mathsf{Smpl} \bigl({G_j},r /\sqrt {2
\edges(G_j)}\bigr) \bigr) \bigr] = r^2/2
\]
for any simple graph $G_j$. However, it is not necessarily true that the
expected number of edges of the limiting graph is $r^2/2$. For example,
consider the case where $G_j$ is a star with $j$ rays. In this case, the
sampled subgraph is non-empty only if the center of the star is selected by
the vertex sampling. The probability that this happens goes to $0$ as $j \to
\infty$, so the limiting graph is the empty graph. The following property
characterizes when the limiting graphex $\W$ satisfies $\EE[e(H_r)] = r^2/2$
for $H_r  \sim \mathrm{GPD}  (\W,r ) $.

\begin{defn} \label{defn:USR}
A sequence of graphs $G_1, G_2, \dots$ is \defnphrase{uniformly sampling
regular} if for all $\epsilon> 0$ there is some $k > 0$ such that, uniformly
for all $j$,
\[
\frac{1}{\edges(G_j)} \sum_{i=1}^{\vertices(G_j)}d_{j,i}1
\bigl[ d_{j,i} > k \sqrt{\edges(G_j)} \bigr]\mathclose {} <
\epsilon,
\]
where $d_{j,i}$ is the degree of vertex $i$ in $G_j$ ignoring loops.
\end{defn}

Intuitively, this property is the requirement that, asymptotically, only a
vanishing fraction of the edges of the graph are due to vertices with
exceptionally high degree. This is a weakening of the condition of uniform
tail regularity: a sequence that is not uniformly sampling regular is also
not uniformly tail regular (see Remark~\ref{rem:tail-vs-sampling-reg} below),
but for example, graph sequences that consist of only isolated edges are
uniformly sampling regular but not uniformly tail regular.

\begin{remark}
\label{rem:tail-vs-sampling-reg} For a sequence of graphs, the sets $U_j$ in
Definition~\ref{def:tail-reg} can without loss of generality be assumed to
correspond to the high degree vertices in $G_j$. Formulated differently, a
sequence of graphs $G_1, G_2, \dots$ has uniformly regular tails iff for each
$\eps>0$ we can find an $M<\infty$ such that when vertices are ordered from
highest to lowest degrees, then
\[
\frac{1}{2\edges(G_j)} \sum_{i> M\sqrt{e(G_j)}}d_i(G_j)
\leq\eps,
\]
for all $j$, where $d_i(G_j)$ denotes the degree of vertex $i$ in $G_j$; see
\cite{Borgs:Chayes:Cohn:Holden:2016}, Remark~18. While this is a statement
about the negligible contribution of the low degree tail of the degree
distribution, it interestingly also implies that vertices of large degrees
only have a negligible contribution; that is, it implies that the sequence
$G_1,G_2,\dots$ is uniformly sampling regular.
\end{remark}

\begin{proof}
For $M<k$, the degree of a vertex of degree at least $k\sqrt{\edges(G)}$
clearly does not change by more than a factor of $(1-M/k)$ if we remove at
most $M\sqrt{\edges(G)}$ of its neighbors from the graph. As a consequence,
\[
\begin{aligned} &\sum_i{d_i(G)}1
\bigl[d_i>k\sqrt{\edges(G)}\bigr]
\\
&\qquad \leq\sum_i 1\bigl[d_i>k
\sqrt {\edges(G)}\bigr] \biggl(\frac{1}{1-M/k } \sum
_{\ell> M\sqrt
{e(G)}}1\bigl[(i,\ell) \in E(G)\bigr] \biggr)
\\
&\qquad \leq\frac{1}{1-M/k}\sum_{\ell> M\sqrt{e(G)}}d_\ell.
\end{aligned} %
\]
With the help of this bound, the proof is straightforward.
\end{proof}

\begin{lem}\label{lem:uniform-int}
Let $G_1,G_2,\dots$ be a graph sequence with $e_j=e(G_j)\to\infty$. Then
$\edges(\mathsf{Smpl} ({G_j},r/\sqrt{2\edges_j}))$ is uniformly integrable
for every $r$ if and only if $G_1, G_2, \dots$ is uniformly sampling regular.
\end{lem}

\begin{proof}
Let $e^r_j \equaldist\edges(\mathsf{Smpl} ({G_j},r/\sqrt{2\edges_j}))$.
Uniform integrability is the statement that for each $\eps>0$ we can find an
$M<\infty$ such that
\[
\limsup_{j\to\infty}\EE\bigl[e^r_j1
\bigl[e^r_j > M\bigr]\bigr]\leq\eps.
\]
Let $d_{j,i}$ denote the degree of vertex $i$ in $G_j$ (ignoring loops, as
usual), and let $D^r_{j,i}$ be the degree in the sampled subgraph of vertex
$i$ in $G_j$, where $D^r_{j,i} = 0$ if vertex $i$ is not selected. Then
$e^r_j=\tfrac{1}{2}\sum_i D^r_{j,i}$. As we will see, the contributions to
this sum that determine whether $e_j^r$ is uniformly integrable come from the
high-degree vertices in $G_j$, specifically from the vertices in a set of the
form $H_j = \{i \in V(G_j) \st d_{j,i} > k\sqrt{e_j}\} $ for a suitable
$k>1$.

To show that uniform sampling regularity is necessary for uniform
integrability, we observe that $D^r_{j,i}$ is given by
\[
D^r_{j,i}= X_{j,i}^r
B^r_{j,i},
\]
where $X_{j,i}^r$ and $B^r_{j,i}$ are independent random variables with
\begin{align*}
X_{j,i}^r \sim \BernDist \biggl(\frac{r}{\sqrt{2e_j}}
\biggr) \quad \text{and}\quad B^r_{j,i} & \sim \binDist
\biggl(d_{j,i},\frac{r}{\sqrt
{2\edges_j}} \biggr).
\end{align*}
(Specifically, $B^r_{j,i}$ is the number of neighbors of vertex $i$ that are
sampled, and $X_{j,i}^r$ is the indicator function for whether $i$ is sampled
itself.) In particular, we can rewrite the sum from Definition~\ref{defn:USR}
as
\[
\frac{1}{\edges(G_j)} \sum_{i=1}^{\vertices(G_j)}d_{j,i}1
\bigl[ d_{j,i} > k \sqrt{\edges(G_j)} \bigr]\mathclose {} =
\frac{2}{r^2}\sum_{i\in H_j}\EE\bigl[D_{j,i}^r
\bigr].
\]
Assume for a moment that for $k$ large and $i\in H_j$,
%
\begin{equation}
\label{B-concentration} \EE\bigl[D_{j,i}^r\bigr]\leq4\EE
\bigl[D_{j,i}^r1\bigl[D^r_{j,i} >
kr/4\bigr] \bigr].
\end{equation}
This would allow us to bound our sum by
\begin{align*}
\frac{2}{r^2}\sum_{i\in H_j}\EE
\bigl[D_{j,i}^r\bigr]&\leq\frac{8}{r^2}\sum
_{i\in H_j}\EE \bigl[D_{j,i}^r1
\bigl[D^r_{j,i} > kr/4\bigr] \bigr]
\\
&\leq\frac{8}{r^2}\sum_{i\in V(G_j)}\EE
\bigl[D_{j,i}^r1\bigl[e^r_j > kr/4
\bigr] \bigr]
\\
&=\frac{8}{r^2}\EE \bigl[e_j^r1
\bigl[e^r_j > kr/4\bigr] \bigr].
\end{align*}
If we assume uniform integrability, the right-hand side can be made
arbitrarily small by choosing $k$ large enough, showing that uniform sampling
regularity is necessary for uniform integrability, once we establish the
bound \eqref{B-concentration}.

To prove \eqref{B-concentration}, we observe that $\Pr (B_{j,i}^r\leq
\frac{1}{2}\EE[B_{j,i}^r] )\leq\exp (-\tfrac{1}8\EE[B_{j,i}^r]
 )$ by the multiplicative Chernoff bound. For $i\in H_j$, we have that
$\EE[B_{j,i}^r]=\frac{rd_{j,i}}{\sqrt{2 e_j}}\geq\frac{kr}{\sqrt {2}}$, so
for $k\geq\frac{8}r$, we have that $\Pr (B_{j,i}^r\leq\frac{1}2\EE[B_{j,i}^r]
)\leq\exp(-1/\sqrt2)\leq\frac{1}2$. Combined with the fact that
$kr/4<\frac{1}2\EE[B_{j,i}^r]$ if $i\in H_j$, this allows us to bound
\begin{align*}
\EE \bigl[D^r_{j,i}1\bigl[D^r_{j,i}
> kr/4\bigr] \bigr] &\ge\EE \biggl[D^r_{j,i}1
\biggl[D^r_{j,i} > \frac{1}2\EE
\bigl[B_{j,i}^r\bigr] \biggr] \biggr]
\\
&=\frac{r}{\sqrt{2e_j}} \EE \biggl[D^r_{j,i}1
\biggl[D^r_{j,i} > \frac{1}2\EE
\bigl[B_{j,i}^r\bigr] \biggr] | X_{j,i}^r=1
\biggr]
\\
&= \frac{r}{\sqrt{2\edges_j}}\EE \biggl[B^r_{j,i}1
\biggl[B^r_{j,i} > \frac{1}2\EE
\bigl[B_{j,i}^r\bigr] \biggr] \biggr]
\\
& \ge\frac{r}{\sqrt{2\edges_j}}\frac{1}2\EE\bigl[B_{j,i}^r
\bigr]\Pr \biggl(B^r_{j,i} > \frac{1}2\EE
\bigl[B_{j,i}^r\bigr] \biggr)
\\
&\ge\frac{r}{\sqrt{2\edges_j}}\frac{1}4\EE\bigl[B_{j,i}^r
\bigr]=\frac{1}4 \EE \bigl[D_{j,i}^r\bigr] ,
\end{align*}
proving \eqref{B-concentration}, and hence the necessity of uniform sampling
regularity.\vadjust{\goodbreak}

To prove that uniform sampling regularity is sufficient to for uniform
integrability, we will need to control various other terms, but it turns out
that the contribution of the vertices in $H_j$ is the only one that requires
uniform sampling regularity. The details are tedious, and are given in the
rest of this proof.

Consider first the number of isolated edges, $e_j^{r,I}$, that is, the number
of edges $\{i,i'\}\in E(G_j)$ such that $D_{j,i}^r=D_{j,i'}^r=1$. The
probability that a given edge in $E(G_j)$ is an isolated edge is then bounded
by $r^2/e_j$, and the probability that two edges $b,b'\in E(G_j)$ are both
isolated is at most $r^4/e_j^2$, that is, the probability that all four
termini are selected by the sampling, except for the case that $b=b'$, in
which case we only have the upper bound $r^2/e_j$. As a consequence, the
expectation of $(e_j^{r,I})^2$ is bounded by $e_j(e_j-1)r^4/e_j^2 + e_j
r^2/e_j\leq r^4+r^2$. Thus $e_j^{r,I}$ is square integrable uniformly in $j$,
and hence uniformly integrable.

Next, given $k > 1$, we partition the vertices of $G_j$ into three sets:
\begin{align*}
H_j &= \bigl\{i \in V(G_j) \st d_{j,i} > k
\sqrt{e_j}\bigr\},
\\
M_j &= \bigl\{i \in V(G_j) \st d_{j,i} \in[
\sqrt{e_j}/k, k\sqrt{e_j}] \bigr\} ,
\\
L_j &= \bigl\{i \in V(G_j) \st d_{j,i} <
\sqrt{e_j}/k\bigr\}.
\end{align*}
We then partition the set of edges contributing to $e_j^r-e_j^{r,I}$ into
several classes, starting with the edges which have one endpoint of degree
$1$ in $L_j$ and one endpoint of degree at least $2$ in $M_j$. Denote the
number of these edges by $e_j^{L,r,1}$, and consider the expectation of
$(e_j^{L,r,1})^2$. We then bound $e_j^{L,r,1}$ by $\sum_{i\in M_j}\sum_{u\in
L_j} X_{iu}$, where $X_{iu} = 1[\{i,u\}\in E(G_j)]1[D_{j,i}^r\geq
2]1[D_{j,u}^r=1]$. Observe that $\EE[X_{iu}X_{i'u'}]\leq r^4/4e_j^2$ if
$i\neq i'$, because each of $i,i',u,u'$ must be selected by the sampling. As
a consequence,
\begin{align*}
\EE \bigl[\bigl(e_j^{L,r,1}\bigr)^2 \bigr] &\leq
\sum_{\substack{i,i' \in M_j\\i \ne i'\\u,u' \in L_j \\ \{i,u\}
, \{i',u'\}\in E(G_j)}} \EE[X_{iu}X_{i'u'}] +
\sum_{\substack{i \in
M_j\\u,u' \in L_j}} \E[X_{iu} X_{iu'}]
\\
&\leq\sum_{\{i,u\}, \{i',u'\}\in E(G_j)} \frac{r^4}{4e_j^2} + \sum
_{i \in M_j} \E \biggl[ \biggl(\sum_{u \in L_j}
X_{iu} \biggr)^2 \biggr]
\\
&\leq\frac{r^4}{4} + \sum_{i\in M_j}\EE\bigl[
\bigl(D_{j,i}^r\bigr)^2\bigr]
\\
&\leq\frac{r^4}{4} + \frac{r}{\sqrt{2e_j}}\sum_{i\in M_j}
\EE \bigl[\bigl(B_{j,i}^r\bigr)^2\bigr]
\\
&\leq\frac{r^4}{4} + \frac{r}{\sqrt{2e_j}}\sum_{i\in M_j}
\biggl(\frac{r^2}{2e_j}d_{j,i}^2+\frac{r}{\sqrt{2e_j}}d_{j,i}
\biggr)
\\
&\le\frac{r^4}{4} + \frac{r}{\sqrt{2e_j}}\sum_{i\in M_j}
\biggl(\frac{r^2k}{2\sqrt{e_j}}d_{j,i}+\frac{r}{\sqrt{2e_j}}d_{j,i}
\biggr),
\end{align*}
where in the last step we used that $d_{j,i}\leq k\sqrt{e_j}$ when $i\in
M_j$. Since $\sum_{i\in M_j}d_{j,i}\leq2e_j$, we see that for each $k$, the
right-hand side is bounded uniformly in $j$, as required.

The remaining contribution to $e_j^r$ will be bounded by
\[
e^{H,r}_j + e^{M,r}_j +
e^{L,r,\geq2}_j,
\]
where $e^{H,r}_j = \sum_{i\in H_j} D^r_{j,i}$, $e^{M,r}_j = \sum_{i\in
M_j}D^{M,r}_{j,i}$, $D_{j,i}^{M,r}$ is the degree of $i$ of edges in subgraph
of the sampled graph $\mathsf{Smpl} ({G_j},r/\sqrt{2\edges_j})$ induced by
restricting to vertices that belong to $M_j$ in $G_j$, and
\[
e^{L,r,\geq
2}_j = \sum_{i\in L_j}
D^r_{j,i}1\bigl[D^r_{j,i}\geq2
\bigr].
\]

Let $M_j^r\subseteq M_j$ be defined by keeping each vertex in $M_j$
independently with probability $r/\sqrt{2 e_j}$. Then $e^{M,r}_j \leq
|M_j^r|^2$. Observing that are at most $2k\sqrt{\edges_j}$ vertices in $M_j$,
since otherwise there would be too many edges, we stochastically bound
$|M_j^r|$ by $\vertices^{M,r}_j  \sim \binDist(2k\sqrt{\edges_j},
r/\sqrt{2\edges_j})$. Since the expectation of $(\vertices^{M,r}_j)^4$ is
bounded uniformly in $j$, this proves that for each $k$, $e^{M,r}_j$ is
square integrable uniformly in $j$, and hence uniformly integrable.

By the assumption of uniform sampling regularity and the fact that
\[
\EE\bigl[e^{H,r}_j\bigr] = \frac{r^2}{2\edges_j} \sum
_{i=1}^{\vertices
_j}d_{j,i}1 [ d_{j,i} > k
\sqrt{\edges_j} ],
\]
we may uniformly force $\EE[e^{H,r}_j]$ to be arbitrarily small by choosing
$k$ sufficiently large.

Finally, direct computation gives that
\[
\EE\bigl[D^r_{j,i} 1\bigl[D^r_{j,i}
\ge2\bigr]\bigr] =\frac{r}{\sqrt{2e_j}} \bigl(\EE \bigl[B^r_{j,i}
\bigr] - \Pr\bigl(B^r_{j,i}=1\bigr) \bigr)\le ({r}/{\sqrt{2
\edges_j}} )^3 d_{j,i}^2,
\]
whereby
\begin{align*}
\EE\bigl[e^{L,r,\geq2}_j\bigr] &\le \biggl(\frac{r}{\sqrt{2}}
\biggr)^3 \sum_{i\in L_j} \frac{d_{j,i}}{\sqrt{\edges_j}}
\frac{d_{j,i}}{\edges_j} \le \biggl(\frac{r}{\sqrt{2}} \biggr)^3
\frac{1}{k} \sum_{i\in L_j} \frac{d_{j,i}}{\edges_j} \le
\frac{r^3}{\sqrt{2}k},
\end{align*}
where the second line has used that $d_{j,i}/\sqrt{\edges_j} < 1/k$ for every
vertex $i$ in $L_j$.

Now, for any constant $c' > 0$,
\begin{align*}
\EE\bigl[e^r_j 1\bigl[e^r_j >
c'\bigr]\bigr]\le {}&\EE\bigl[(\bigl(e_j^{r,I}+e_j^{r,L,1}
+ e^{M,r}_j\bigr)1\bigl[(\bigl(e_j^{r,I}+e_j^{r,L,1}
+ e^{M,r}_j\bigr) > c'/2\bigr]\bigr]
\\
& {} + \EE\bigl[(\bigl(e_j^{r,I}+e_j^{r,L,1}
+ e^{M,r}_j\bigr)1\bigl[e^{L,r,\geq2}_j +
e^{H,r}_j > c'/2\bigr]\bigr]
\\
&{} + \EE\bigl[e^{L,r,\geq2}_j\bigr] + \EE\bigl[e^{H,r}_j
\bigr].
\end{align*}
For $\epsilon> 0$, we may guarantee that the last two terms are each at most
$\epsilon/4$ by choosing $k$ sufficiently large. For any fixed $k$, Markov's
inequality shows that $\lim_{c'\to\infty}\Pr(e^{L,r,\geq2}_j + e^{H,r}_j >
c'/2) = 0$. \cite{Kallenberg:2002}, Lemma~4.10, shows that for any uniformly
integrable family $\{X_j\}$ and sequence of events $A_1, A_2,\dots$ such that\break
$\lim_{k\to\infty}\Pr(A_k) = 0$, we have
$\lim_{k\to\infty}\sup_j\EE[X_j1[A_k]] = 0$; accordingly, invoking the
uniform integrability of $e_j^{r,I}+e_j^{r,L,1} + e^{M,r}_j$, we may choose
$c'$ (depending on $k$) large enough such that the second term is at most
$\epsilon/4$. Similarly, by uniform integrability, we may choose $c'$ large
enough such that the first term is at most $\epsilon/4$. Thus, for any
$\epsilon$ there is a $c' > 0$ such that
\[
\EE\bigl[e^r_j 1\bigl[e^r_j >
c'\bigr]\bigr] < \epsilon
\]
uniformly, as required.
\end{proof}

\begin{cor}\label{integ_of_lim}
The limiting graphex $\W=(I,S,W)$ in Lemma~\ref{lim_is_graphex} is
integrable, with $\|\W\|_1\leq1$ and $\int W(x,x) \,\intd{x} =
\lim_{j\to\infty}\loops(G_j)/\sqrt{2\edges(G_j)}$. Further, the bound is
saturated if and only if the graph sequence is uniformly sampling regular.
\end{cor}

\begin{proof}
Let $\xi$ be the limiting point process as in the proof of
Lemma~\ref{lim_is_graphex}, let $\KEG_r=\xi|_{[0,r]^2}$, let
$\loops_r=\ell(\KEG_r)$ and let $e^r$ and $e^r_j$ be defined by
$e^r=\edges(\xi(\cdot\cap[0,r)^2))$ and $e^r_j=\edges(\mathsf
{Lbl}(G_j)(\cdot \cap[0,r)^2))$.

Observe that $\loops(\mathsf{Smpl} ({G_j},r/\sqrt{2\edges_j}))  \sim
\binDist(\loops(G_j),r/\sqrt{2\edges_j})$, and that loops in the sampled
subgraph can only occur by selecting loops in the original graph. It then
follows that $\EE[\loops_r] =
r\lim_{j\to\infty}\loops(G_j)/\sqrt{2\edges(G_j)}$ for all $r\in \Reals_{+}$.
Comparing this expression with Lemma~\ref{expected_edges} establishes the
claim about the diagonal part of~$W$.

We have $\EE[e^r] \le\lim_{j\to\infty}\EE[e^r_j] = \frac {1}{2}r^2$ by a
version of Fatou's lemma \cite{Kallenberg:2002},\break Lemma~4.11. Comparing with
Lemma~\ref{expected_edges} establishes that $\|\W\|_1 \le1$.

The second claim follows from the observation that $\edges(\mathsf
{Lbl}(G_j)(\cdot \cap[0,r)^2))\equaldist\edges(\mathsf{Smpl}
({G_j},r/\sqrt{2\edges_j}))$, Lemma~\ref{lem:uniform-int}, and the fact that
a sequence of non-negative random variables $X_1, X_2, \dots$ that converges
in distribution to $X$ also satisfies $\EE[X_j] \to\EE[X]$ if and only if it
is uniformly integrable.
\end{proof}

We now have the ingredients of the main result characterizing the limits of
sampling convergent sequences.

\begin{theorem}\label{lim_is_det_int_graphex}
Let $G_1,G_2, \dots$ be a sampling convergent graph sequence such that
$\edges(G_j) \to\infty$ as $j\to\infty$. Then the limit is a nonrandom
graphex $\W$ such that $\|\W\|_1 \le1$, in the sense that $\mathrm {SmplD}
(G_j,r ) \to\mathrm{GPD}  (\W,r ) $ weakly as $j \to\infty$ for all $r
\in\Reals_{+}$. The bound on $\|\W\|_1$ is saturated if and only if the
sequence is uniformly sampling regular.
\end{theorem}

\begin{proof}
Immediate from Lemmas~\ref{lim_is_graphex} and \ref{lim_is_deter} and
Corollary \ref{integ_of_lim}.
\end{proof}

In some other sparse graph limit theories \cite
{Borgs:Chayes:Cohn:Zhao:2014:sgc1,Borgs:Chayes:Cohn:Zhao:2014:sgc2,Borgs:Chayes:Cohn:Holden:2016},
only graph sequences satisfying certain constraints are subsequentially
convergent. We prove a compactness result in Section~\ref{sec:metrization}
that has the following corollary.

\begin{theorem}\label{subseq_conv}
Every sequence of graphs $G_1, G_2, \dots$ satisfying $\loops(G_j) =
O(\sqrt{\edges(G_j)})$ is subsequentially sampling convergent.
\end{theorem}

\begin{proof}
This will be immediate from Theorem~\ref{theorem:compact}.
\end{proof}

On the basis of this result, one might hope that sampling convergent limits
are informative about all sparse graph sequences, or at least all uniformly
sampling regular sequences. The next result helps clarify that there are
further limitations. Intuitively speaking, it shows that the sampling limit
is degenerate for sparse graph sequences with relatively homogeneous degrees.
In particular, the next result applies to sequences of bounded degree graphs,
for which there is already a well developed limit theory
\cite{Benjamini:Schramm:2001:1}. It also applies to the random graph
$G_{n,p}$ as long as $p\to0$ and $n^2p\to\infty$ as $n\to\infty$, or more
generally, to inhomogeneous random graphs obtained by first choosing a dense
random graph sequence generated by a bounded graphon and then subsampling it
so that it becomes sparse, again as long as it is dense enough to guarantee
that the number of edges goes to infinity a.s.

To state the theorem, we define the average degree and square average degree
of a graph $G$ as $\overline{d}{(G)}=\frac{1}{\vertices(G)}\sum_{i}d_i(G)$
and $\overline{d^2}(G)=\frac{1}{\vertices(G)}\sum_{i}(d_i(G))^2$, where
$d_i(G)$ is the degree of vertex $i$ not counting loops. We also recall that
the edge density of $G$ is defined as
$\rho(G)=2\edges(G)/(\vertices(G))^2=\overline{d}(G)/\vertices(G)$.

\begin{theorem}\label{ultra_sparse_lim_triv}
Let $G_1,G_2,\dots$ be a sampling convergent graph sequence with $\edges(G_j)
\to\infty$ as $j\to\infty$. Suppose that the maximal degree of $G_j$ is
$o(\sqrt{\edges(G_j)})$ or, more generally, that
%
\begin{equation}
\label{deg-cond} \frac{\overline{d^2}(G)}{(\overline{d}(G))^2}\sqrt{\rho(G_j})=o(1).
\end{equation}
Then $G_1,G_2,\dots$ is sampling convergent to a graphex of the form
$(1/2,0,W)$, where the graphon $W$ is zero except on the diagonal.
\end{theorem}

\begin{proof}
Let $r \in\Reals_{+}$. For brevity, let $\vertices_j = \vertices(G_j)$,
$\edges_j=\edges(G_j)$ and $p_j = {r}/{\sqrt{2e_j}}$. Let $d_{j,i}$ be the
degree of vertex $i$ in $G_j$ and let $D^r_{j,i}$ be the degree of this
vertex in a $p_j$-sampled subgraph, where $D^r_{j,i} = 0$ is understood to
mean that the vertex is not included in the subgraph.

We first prove that the assumption \eqref{deg-cond} implies uniform sampling
regularity. To this end, we bound
\begin{align*}
\frac{1}{\edges(G_j)} \sum_{i=1}^{\vertices(G_j)}d_{j,i}1
\bigl[ d_{j,i} > k \sqrt{\edges(G_j)} \bigr] &\leq
\frac{1}{k(\edges(G_j))^{3/2}} \sum_{i=1}^{\vertices
(G_j)}(d_{j,i})^2
\\
&=2 \frac{\sqrt2}{k} \frac{\overline{d^2}(G_j)\sqrt{\rho(G_j)}}{
(\overline{d}(G_j))^{2}},
\end{align*}
from which the claim follows.

Next, we recall that
\[
D^r_{j,i} \given B_{j,i} \sim
(1-p_j)\delta_0 + p_j \delta_{B_{j,i}}
\qquad \text{where } B_{j,i} \sim \binDist(d_{j,i},p_j),
\]
so in particular
\begin{align*}
\Pr\bigl(D^r_{j,i} \ge2\bigr) &= p_j\bigl(1-
\bigl[(1-p_j)^{d_{j,i}} + d_{j,i} p_j
(1-p_j)^{d_{j,i}-1}\bigr]\bigr) \le p_j^3
d_{j,i}^2,
\end{align*}
using Bernoulli's inequality. Let $N_j$ be the number of vertices with degree
greater than $1$ in the sampled subgraph. Then
\begin{align*}
\EE[N_j] &\le\sum_{i \le\vertices_j}\Pr
\bigl(D^r_{j,i} \ge2\bigr) \le p_j^3
\sum_{i \le\vertices_j} d_{j,i}^2
=r^3 \frac{\overline{d^2}(G)}{(\overline{d}(G))^2}\sqrt{\rho(G_j)} = o(1).
\end{align*}
Markov's inequality then implies that $N_j \convPr0$ as $j \to\infty$. Since
$r$ was arbitrary, this implies convergence to a graphex of the claimed form.
\end{proof}

As a corollary of the theorem, the limit of a sequence of preferential
attachment graphs is the pure edge graphex. More generally, we have the
following corollary.

\begin{cor} Let $G_1,G_2,\dots$ be a random sequence of simple graphs
such that almost surely (\textup{a}) the empirical degree distribution
converges to a distribution with finite, positive mean, (\textup{b}) the
average degree converges to the mean of the limiting degree distribution and
(\textup{c}) $\limsup_{j\to\infty} \frac {\max_i
d_i(G_j)}{\sqrt{\vertices(G_j)}}<\infty$ and
$\lim_{j\to\infty}\edges(G_j)=\infty$. Then a.s., $G_1,G_2,\dots$ is sampling
convergent to the graphex $(1/2,0,0)$.
\end{cor}

\begin{proof}
Let $P_d$ be the limit of the probability that a random vertex in $G_j$ has
degree $d$, let $\overline{d}$ be the mean of $P$ and let $d_{j,i}$ and
$\vertices_j$ be as in the last proof. Then
\[
\begin{aligned} \lim_{j\to\infty}\frac{1}{\vertices_j}\sum
_i d_{j,i}1[d_{j,i}\geq k] &
\equalas \overline{d}-\lim_{j\to\infty}\frac{1}{\vertices_j}\sum
_i d_{j,i}1[d_{j,i}<k]
\\
&\equalas\overline{d}-\sum_{d<k}dP_d=
\sum_{d\geq k} dP_d. \end{aligned}
\]

Given $\eps>0$, let $k$ be a (possibly random) finite constant such that the
right-hand side is at most $\eps/2$, and let $J<\infty$ be such that for
$j\geq J$,
\[
\frac{1}{\vertices_j}\sum_i d_{j,i}1[d_{j,i}
\geq k]\leq\eps.
\]

Defining $C_j=\frac{1}{\sqrt{\vertices_j}}\max_i d_{j,i}$, we then have that
\[
\frac{1}{\vertices_j}\sum_i d_{j,i}^2
\leq \frac{C_j}{\sqrt{\vertices_j}}\sum_i d_{j,i}1[d_{j,i}
\geq k]+\frac{k}{\vertices_j}\sum_i d_{j,i}
\leq\eps C_j\sqrt{\vertices_j}+k\overline{d}(G_j).
\]
Using that $\rho(G_j)=\overline{d}(G_j)/\vertices_j$, this shows that
\[
\frac{\sqrt{\rho(G_j)}}{(\overline{d}(G_j))^2}\overline{d^2}(G_j) \leq
\overline{d}(G_j)^{-3/2} \eps C_j+
\frac{k}{\sqrt{\vertices
_j\overline{d}(G_j)}}.
\]
Recalling that $\vertices_j\overline{d}(G_j)=2\edges(G_j)$, we can now first
take the limit superior over $j$ and then the limit $\eps\to0$ to see that
the condition \eqref{deg-cond} is a.s. satisfied. To complete the proof, we
use that every sequence of loopless graphs $G_1,G_2,\dots$ with
$\edges(G_j)\to\infty$ has a convergent subsequence.
\end{proof}

\section{Graphex processes are sampling convergent}\label{sec:gp_samp_conv}

We now turn to characterizing the sampling limits of sequences of graphs
generated by a graphex process. Let $s_1, s_2, \dots$ be some sequence such
that $s_j \upto\infty$ as $j \to\infty$ and let $G_j = \mathcal{G}
 (\KEG_{s_j} ) $,
where $\KEG$ is generated by an integrable graphex $\W$. Intuitively
speaking, our aim is to show that the sampling limit of $G_1, G_2, \dots$ is
$\W$.

The basic strategy makes use of the consistent estimation results first
established in \cite{Veitch:Roy:2016}, although we will appeal to the
technically stronger versions of \cite{Janson:2017}. We need the following
(implicit) result from those papers.

\begin{lem}\label{r_s_samp_conv}
Let $G_s = \mathcal{G}  (\KEG_{s} ) $, where $(\KEG _s)_{s\in\Reals_{+}}$ is
generated by an integrable graphex $\W$, then $\Pr(\mathsf{Smpl} (G_s, r /s)
\in\cdot \given G_s) \to\mathrm{GPD}  (\W,r ) $ weakly almost surely as $s
\to \infty$, for all $r \in \Reals_{+}$.
\end{lem}

\begin{proof}
Let $\widehat{W}_{(G_s,s)}$ be the empirical graphon of $G_s$ stretched so
that each pixel is $1/s \times1/s$. \cite{Janson:2017}, Theorem~5.1, shows
that $\mathrm{GPD}  (\widehat{W}_{(G,s)},r ) \to\mathrm {GPD}  (\W,r ) $
weakly almost surely. As noted earlier, $\mathrm{GPD}
(\widehat{W}_{(G_s,s)},r ) = \Pr(\mathsf{SmplWR} ({G_s},\break  r/{s})
\in\cdot \given G_s) $, and so the result follows from the asymptotic
equivalence of with and without replacement sampling,
Lemma~\ref{lem:coupling}. Indeed, for each fixed $r$, we have that a.s.
$\edges(G_s) (r/s)^3\to0$ and $\loops(G_s)(r/s)^2\to0$ as $s\to \infty$ (by,
e.g., Lemma~\ref{edge_scaling} below). Lemma~\ref{lem:coupling} then implies
that conditioned on $(\KEG_s)_{s\in\Reals_{+}}$, the total variation distance
between the with and without replacement distributions goes to zero a.s. as
$s\to\infty$.
\end{proof}

To drop the latent {time}s, we will need an extension of a result of
\cite{Borgs:Chayes:Cohn:Holden:2016} relating $\edges(G_j)$ and $s_j$. It
will be convenient to partition each $\KEG_s$ into three components, which
correspond to the three terms in \eqref{KEGgen}. We will use the notation of
the Kallenberg representation theorem
(Theorem~\ref{theorem:graphex_rep_theorem}). Let $\PP$ be the latent Poisson
process used in the Kallenberg representation construction, and let $\PP_s$
be the restriction of $\PP$ to $[0,s) \times \Reals_{+}$. We partition
$\KEG_s$ into the following three pieces, corresponding to the three terms in
the representation theorem:
\begin{enumerate}
\item[1.]$\KEG_s^W$: the edge induced subgraph given by restricting to edges
    between vertices that belong to the underlying Poisson process $\PP_s$;
    this is the part of the graph generated by $(0,0,W)$.
\item[2.]$\KEG_s^S$: the edge induced subgraph given by restricting to edges
    where one vertex belongs to any latent star Poisson process
    $\sigma_{jk}$; this is the part of the graph generated by $(0,S,0)$.
\item[3.]$\KEG_s^I$: the induced subgraph given by restricting to the
    remaining edges; this is the part of the graph generated by $(I,0,0)$.
\end{enumerate}

\begin{lem}\label{edge_scaling}
Let $(\KEG_s)_{s\in\Reals_{+}}$ be a graphex process generated by graphex
$\W=(I,S,W)$, and let $\edges^W_{s}$, $\edges^S_{s}$ and $\edges ^I_{s}$ be
the number of edges of $\KEG_{s}^W$, $\KEG_{s}^S$ and $\KEG_{s}^I$,
respectively. Then, almost surely,
\begin{align*}
\lim_{s\to\infty}\edges^W_{s}/s^2
&= \frac{1}{2} \llVert W \rrVert _1,\qquad \lim
_{s\to\infty}\edges^S_{s}/s^2 =
\llVert S \rrVert _1,\qquad \lim_{s\to\infty}
\edges^I_{s}/s^2 = I,\quad \text{and}
\\
\lim_{s\to\infty}\loops(\KEG_s)/s &=
\int W(x,x) \,\intd{x}.
\end{align*}
\end{lem}

\begin{proof}
First, $\edges^W_{s} / s^2 \to\frac{1}{2}\|W\|_1 \as$ by
\cite{Borgs:Chayes:Cohn:Holden:2016}, Proposition~30.

The case $\|S\|_1 = 0$ is trivial. Assume $\|S\|_1 > 0$. The star component
of the graphex process can be understood as assigning a $\poiDist(sS(x_i))$
number of rays to each point of the underlying point process $(t_i,x_i)
\in\PP_{s}$, independent of everything else. By the additive property of
independent Poisson distributions, we then have $\edges^S_{s} \given\PP_{s}
\sim \poiDist(s\sum_{(t_i,x_i) \in \PP_{s}}S(x_i))$. Since $s \sum_{(t_i,x_i)
\in\PP_{s}}S(x_i) \upto\infty\as$ as $s \to\infty$, the law of large numbers
implies $\edges^S_{s} / (s \sum_{(t_i,x_i) \in \PP_{s}}S(x_i)) \to1 \as$ as
$j\to\infty$. The law of large numbers for Poisson processes gives
$\sum_{(t_i,x_i) \in\PP_{s}}S(x_i) / s \to\|S\|_1 \as$ as $j \to\infty$,
whereby $\edges^S_{s} / s^2 \to\|S\|_1 \as$ as $j\to\infty$.

We have $\edges^I_{s} / s^2 \to I \as$ as $s \to\infty$ by the law of large
numbers for Poisson processes.

Finally, we may view the loops as an independent marking of the latent
Poisson process, with a loop on $(t_i,x_i)$ included with probability
$W(x_i,x_i)$. The fact that $\lim_{s\to\infty}\loops(\KEG_s)/s = \int W(x,x)
\,\intd{x}$ then follows by the law of large numbers for Poisson processes.
\end{proof}

By the two previous lemmas, the limiting distribution of $\mathsf {Smpl}
({G_s},r /s)$ is generated by $\W$, and (temporarily simplifying to the case
$\|W\|_1 + 2\|S\|_1 + 2I = 1$) we have $s \approx\sqrt{2 \edges(G_s)}$ when
$s$ is large. Thus, to prove our main result we would like to couple
$\mathsf{Smpl}  ({G_s},r /s)$ and $\mathsf{Smpl} ({G_s},r/{\sqrt{2
\edges(G_s)}})$.

\begin{theorem}\label{graphex_proc_samp_conv}
Let $(\KEG_s)_{s\in\Reals_{+}}$ be a graphex process generated by an
integrable graphex $\W=(I,S,W)$ such that $\|\W\|_1 > 0$, and let $G_s =
\mathcal{G}  (\KEG_{s} ) $ for all $s \in\Reals_{+}$. Then
$(G_s)_{s\in\Reals_{+}}$ is sampling convergent to $\W'$, that is,
$\mathrm{SmplD}  (G_s,r ) \to\mathrm {GPD}  (\W',r ) $ weakly almost surely,
where $\W'=(I',S',W')$ is defined by
\begin{align*}
I' &= I/ \llVert \W \rrVert _1,\qquad
S'(x)= \bigl( \llVert \W \rrVert _1
\bigr)^{-1/2} S\bigl(x \llVert \W \rrVert _1^{1/2}
\bigr), \quad \text{and }
\\
W'(x,y) &= W\bigl(x \llVert \W \rrVert _1^{1/2},
y \llVert \W \rrVert _1^{1/2}\bigr).
\end{align*}
\end{theorem}

\begin{proof}
First, for any graph $G$ and any $q, p \in[0,1]$ such that $q < p$, there is
a coupling such that
\[
\Pr\bigl(\mathsf{Smpl} (G,p) \neq\mathsf{Smpl} (G ,q)\bigr) \le\bigl(2
p^2 \edges(G) + p \loops (G)\bigr) (1-q/p).
\]
Explicitly, we sample $\mathsf{Smpl} (G,p)$ as usual, and we sample
$\mathsf{Smpl} (G ,q)$ as\break
$\mathsf{Smpl} (\mathsf{Smpl} (G,p),q /p)$. Then
the expected number of vertices included in\break
$\mathsf{Smpl} (G,p)$ that are
not selected as candidates for $\mathsf{Smpl} (G ,q)$ is
\begin{align*}
\EE\bigl[\vertices\bigl(\mathsf{Smpl} (G,p)\bigr)\bigr](1-{q}/{p}) \leq{}&
\bigl(2\EE\bigl[\edges\bigl(\mathsf{Smpl} (G,p)\bigr)\bigr]
\\
&{} + \EE\bigl[\loops\bigl(\mathsf{Smpl} (G,p)\bigr)\bigr]\bigr) (1-{q}/{p})
\\
={}& \bigl(2 p^2 \edges(G) + p \loops(G)\bigr) (1-q/p),
\end{align*}
and the claimed inequality follows by Markov's inequality and the observation
that $\mathsf{Smpl} (G,p) = \mathsf {Smpl} (G ,q)$ if every vertex of
$\mathsf{Smpl} (G,p)$ is included as a candidate for $\mathsf{Smpl} (G ,q)$.

Let $c=\|\W\|_1^{-1/2}$. Under the above coupling,
\begin{eqnarray*}
&&\Pr\biggl(\mathsf{Smpl} \biggl( {G_s},\frac{r}{\sqrt{2 \edges
(G_s)}}\biggr)
\neq\mathsf{Smpl} \biggl({G_s}, \frac{r c}{s}\biggr)\biggr)
\\*
&&\qquad \le %
\begin{cases} \biggl(2 r^2c^2
\frac{\edges(G_s)}{s^2} + rc \frac{\loops
(G_s)}{s} \biggr) \biggl(1-\frac{s/c}{\sqrt{2\edges(G_s)}}
\biggr)
\\
\qquad \text{for } s/c < \sqrt{2 \edges(G_s)}, \text{ and}
\\
\biggl(r^2 + r \frac{\loops(G_s)}{\sqrt{2 \edges(G_s)}} \biggr) \biggl(1-
\frac{\sqrt{2\edges(G_s)}}{s/c} \biggr)
\\
\qquad \text{for } s/c \ge\sqrt{2 \edges(G_s)}. \end{cases}
\end{eqnarray*}
By Lemma~\ref{edge_scaling}, the right-hand side goes to $0$ almost surely as
$s\to\infty$. The theorem statement then follows by Lemma~\ref{r_s_samp_conv}
and \cite{Veitch:Roy:2016}, Lemma~5.2, which implies that $\mathrm {GPD}  (\W
,rc ) = \mathrm{GPD}  (\W^c,r ) $, where $\W^c = (c^2 I, c S(\cdot/ c),
W(\cdot/ c, \cdot/ c))$.
\end{proof}

\begin{cor}\label{every_graphex_lim}
For any integrable graphex $\W$ such that $\|\W\|_1 \le1$ there is some graph
sequence that is sampling convergent to $\W$.
\end{cor}

\begin{proof}
Suppose $\|\W\|_1 = 1$, and let $s_1, s_2, \dots$ be some sequence such that
$s_j \upto\infty$ as $j \to\infty$ and let $G_j = \mathcal{G}
 (\KEG _{s_j} ) $, where
$\KEG$ is generated by $\W$; the sequence $G_1, G_2, \dots$ is almost surely
sampling convergent to $\W$ by Theorem~\ref{graphex_proc_samp_conv}.

Next, suppose that $0<\|\W\|_1<1$, and as above, let $G_j = \mathcal{G}
(\KEG_{s_j} ) $, with $\KEG$ generated by $\W$, and let $S_1, S_2, \dots$ be
a sequence of stars such that $\edges(S_j) \to\infty$ as $j\to \infty$ and
$\lim_{j\to\infty}{\edges(G_j)}/(\edges(G_j)+\edges(S_j)) = \frac{1}2
\|\W\|_1$. Under the obvious coupling,
\[
\lim_{j\to\infty} \mathsf{Smpl} \biggl({G_j \cup
S_j}, \frac{r}{\sqrt
{\edges(G_j \cup S_j)}}\biggr) \equalas\lim_{j\to\infty}
\mathsf{Smpl} \biggl( {G_j},\frac
{r}{\sqrt{\edges(G_j \cup S_j)}}\biggr),
\]
because the probability of seeing even a single edge sampled from $S_j$ is
bounded by the probability of selecting the center of the star as a candidate
vertex, which tends to $0$. By Lemma~\ref{edge_scaling},
$\edges(G_j)/s_j^2\to\frac{1}2 \|W\|_1$ a.s. as $j\to\infty$, implying that
$\edges(G_j\cup S_j)/s_j^2\to1$ a.s. as $j\to\infty$. By essentially the same
coupling argument used in the proof of Theorem~\ref{graphex_proc_samp_conv},
$ \mathrm{SmplD}  (G_j \cup S_j,r ) \to\mathrm{GPD} (\W,r ) $ weakly as $j
\to\infty$, showing that $G_1 \cup S_1, G_2 \cup S_2, \dots$ is sampling
convergent to~$\W$.\looseness=1

Next, consider a sequence $G_1, G_2, \dots$ generated by a graphon $W$ that
is $0$ except on the diagonal, and take $\edges(S_j) = \lceil
 (\loops(G_j)/\int W(x,x)\,\intd{x} )^2 \rceil$. By
Lemma~\ref{edge_scaling} and the fact that $\edges(G_j)=0$ a.s., we see that
$\edges(G_j\cup S_j)/s_j^2\to1$ a.s., showing that $G_1 \cup S_1, G_2 \cup
S_2, \dots$ is sampling convergent to $(0,0,W)$.\vadjust{\goodbreak}

Finally, the sampling limit of $S_1, S_2, \dots$ with $\edges(S_j) = j$ is
$(0,0,0)$, completing the proof.
\end{proof}

\section{Graphon metrics and sampling distributions}\label
{sec:met_and_samp_dist}

In this section, we relate sampling convergence to the metric convergence of
\cite{Borgs:Chayes:Cohn:Holden:2016}. Intuitively, the basic idea is to show
that if $\delta_1(W_1,W_2)$ or $\delta_\square(W_1,W_2)$ is small then we can
construct a coupling of $\mathrm{GPD}  (W_1,r ) $ and $\mathrm{GPD}
 (W_2,r ) $ such that $\Pr(G^1_r
\neq G^2_r)$ is also small, where $G^k_r  \sim \mathrm{GPD} (W_k,r ) $
marginally. Note that we require the diagonals to be $0$ throughout because
the graphon metrics do not control distance between diagonals. Similarly, we
assume the graphons are integrable, since otherwise the metrics are not
defined.

\begin{lem}\label{small_one_close_dist}
Let $W$ and $W'$ be integrable graphons with vanishing diagonals, and let
$H^{(1)}_r\sim \mathrm{GPD}  (W_1,r ) $ and $H^{(2)}_r\sim\mathrm{GPD}
 (W_2,r ) $. Then there is a
coupling of $H^{(1)}_r$ and $H^{(2)}_r$ such that under this coupling
\[
\Pr\bigl(H^{(1)}_r \neq H^{(2)}_r
\bigr) \le\frac{1}2 r^2 \delta_1(W_1,W_2).
\]
\end{lem}

\begin{proof}
We couple $H^{(1)}_r$ and $H^{(2)}_r$ according to the following generative
scheme:
\begin{enumerate}
\item[1.] Draw $\PP \sim \PPDist([0,r) \times\Reals_{+}, \lambda \otimes
    \lambda)$.
\item[2.] Draw U-array $\{U_{ij}\}$.
\item[3.] Include edge $(t_i,t_j)$ in graph $H^{(k)}_r$ if and only if
    $W_k(x_i,x_j) > U_{ij}$.
\item[4.] Drop the labels of the graphs.
\end{enumerate}
That is, we generate both graphon processes using the same latent Poisson
process and U-array. Marginally, this is just the standard graphon process
scheme and so the coupling is obviously valid.

Under this coupling, for each pair of points $(t_i,x_i)$ and $(t_j,x_j)$ in
$\PP$ the probability, conditional on $\PP$, that $(t_i,t_j)$ is an edge in
one graph and not an edge in the other is $|W_1(x_i,x_j)-W_2(x_i,x_j)|$. The
expected number of edges that disagree between the two graphs is then
\[
\frac{1}{2} \EE\biggl[\sum_{x_i,x_j \in\PP
} \bigl\llvert
W_1(x_i,x_j)-W_2(x_i,x_j)
\bigr\rrvert \biggr]=\frac{r^2}{2} \llVert W_1-W_2
\rrVert _1,
\]
where the expectation is computed by an application of the Slivnyak--Mecke
theorem.

The graphs are equal if there are no edges that disagree, so Markov's
inequality then gives
\[
\Pr\bigl(H^{(1)}_r \neq H^{(2)}_r
\bigr) \le\frac{r^2}{2} \llVert W_1-W_2 \rrVert
_1.
\]
For any measure-preserving transformation $\phi$ of $\Reals_{+}$,
$\mathrm{GPD}  (W \circ (\phi\otimes\phi),r ) = \break \mathrm{GPD}  (W,r ) $. It
then follows that
\[
\Pr\bigl(H^{(1)}_r \neq H^{(2)}_r
\bigr) \le\min_{\phi_1,\phi_2}\frac
{r^2}{2} \bigl\llVert
W_1 \circ(\phi_1 \otimes\phi_1) -
W_2 \circ(\phi_2 \otimes\phi_2) \bigr\rrVert
_1,
\]
where the minimization is over all pairs of measure-preserving
transformations.
\end{proof}

To show that convergence in stretched cut distance implies convergence of the
laws of the graphs generated by the graphons, we will need a translation of
the corresponding result (\cite{Borgs:Chayes:Lovasz:Sos:Vesztergombi:2008},
Theorem~3.7a) from the theory of dense graph convergence.

\begin{lem}\label{compact_cut_conv_to_left_conv}
Let $W_1, W_2, \dots$ be a sequence of integrable graphons with vanishing
diagonals. Suppose that there is some compact set $C$ such that $\supp(W_j)
\subseteq C$ for all $j\in\Nats$. If $\lim_{j\to\infty}\delta_\square(W_j,W)
= 0$ for some graphon $\W$, then there is a sequence of couplings of
$\mathrm{GPD} (W_j,r ) $ and $\mathrm{GPD}  (W,r ) $ such that, for
$H^{(j)}_r$ and $H_r$ distributed according to $\mathrm{GPD}  (W_j,r ) $ and
$\mathrm{GPD} (W,r )$, respectively,
\[
\lim_{j\to\infty}\Pr\bigl(H^{(j)}_r \neq
H_r\bigr) = 0 \qquad \as
\]
\end{lem}

\begin{proof}
Because $C$ is compact, $C \subseteq[0,c]^2$ for some $c\in\Reals_{+} $. We
only require a $C$ such that $\supp(W_j) \subseteq C$, and hence we may
assume without loss of generality that $C = [0,c]^2$.

The first ingredient of the coupling is the observation that a sample from\break
$\mathrm{GPD}  (W,r ) $ may be generated according to the following scheme:
\begin{enumerate}
\item[1.] Sample $N_r  \sim \poiDist(cr)$.
\item[2.] For $i=1,\dots ,N_r$ sample features $x_k \distiid U[0,c]$.
\item[3.] Include each edge $(k,l)$ independently with probability
    $W(x_k,x_l)$.
\item[4.] Drop the labels in $[N_r]$ and return the edge set.
\end{enumerate}
That is, in the compactly supported graphon case, the edges are sampled
independently conditional on the number of candidate vertices. This is
essentially the same generative model as is used in the dense graph theory,
with the distinction that the number of vertices is now random and that
vertices that do not connect to any edges are not included in the graph. Our
aim is to build a sequence of couplings that exploits this observation along
with the equivalence of left convergence and cut convergence in the dense
graph setting.

Using a common $N_r$ for sampling from each $W_j$ allows us to use results
from the dense graph setting.
\cite{Borgs:Chayes:Lovasz:Sos:Vesztergombi:2008}, Theorem~3.7a, shows that if
$\delta_\square(W_j,W) \to0$ as $j \to\infty$, then for each fixed graph $F$,
\[
\lim_{j\to\infty} \bigl\lvert\Pr\bigl(H^{(j)}_r
= F \given N_r\bigr) - \Pr (H_r = F \given
N_r) \bigr\rvert = 0.
\]
It is immediate that the limit is also $0$ unconditionally; that is,
\[
\mathrm{GPD} (W_j,r ) \to\mathrm{GPD} (W,r )
\]
weakly as $j \to\infty$. Since the space of graphs is discrete, weak
convergence also implies convergence in total variation. This implies the
existence of the sequence of couplings in the lemma statement.
\end{proof}

The next result extends this to the case of arbitrary cut convergent graphon
sequences. The same result has recently been independently proved as
\cite{Janson:2017}, Theorem~3.4.

\begin{lem}\label{cut_implies_sampling_lemma}
Let $W_1, W_2, \dots$ be a sequence of integrable graphons with vanishing
diagonals such that $\delta_\square(W_j,W) \to0 \as$ as $j \to\infty$ for
some integrable graphon $W$ with vanishing diagonal. Then there is a sequence
of couplings such that, given $H^{(j)}_r$ and $H_r$ distributed according to
$\mathrm{GPD}  (W_j,r ) $ and $\mathrm{GPD}
 (W,r )$, respectively,
\[
\lim_{j\to\infty}\Pr\bigl(H^{(j)}_r \neq
H_r\bigr) = 0.
\]
\end{lem}

\begin{proof}
If the sequence is compactly supported then the result follows from
Lemma~\ref{compact_cut_conv_to_left_conv}, so assume otherwise.

It suffices to show that for all $\epsilon>0$ there is a sequence of
couplings (indexed by $j$) such that there is some $j'$ such that for all
$j>j'$,
\[
\Pr\bigl(H^{(j)}_r \neq H_r \given
W_j\bigr) \le\epsilon.
\]

The basic structure of our couplings is to pick out compactly supported
``dense cores'' of $W$ and $W_j$ such that, with high probability, every edge
of $H^{(j)}_r$ and $H_r$ is due to the dense cores, and then couple these
cores by Lemma~\ref{compact_cut_conv_to_left_conv}. We control the error
introduced by restricting to the dense cores by
Lemma~\ref{small_one_close_dist}.

Because $\delta_\square(W_j,W) \to0 \as$ as $j \to\infty$, we can find a
sequence of measure-preserving maps $\phi_j\colon\R_+\to\R_+$ such that
$\|W^{\phi_j}_j-W\|_\square\to0$. Replacing $W_j$ by $W_j^{\phi _j}$, we may
therefore assume without loss of generality that $\|W_j-W\|_\square\to0$.
Since $W$ is integrable, we can find a constant $M_{r,\epsilon}$ such that
$\|W-W1_{[0,M_{r,\epsilon}]^2}\|_1\leq\eps r^{-2}$. Next, we observe that
\begin{align*}
& \llVert W_j-W_j1_{[0,M_{r,\epsilon}]^2} \rrVert
_1
\\
&\qquad =
\int(W_j-W_j1_{[0,M_{r,\epsilon}]^2})
\\
&\qquad =
\int(W-W1_{[0,M_{r,\epsilon}]^2}) +
\int(W_j-W) +
\int (W_j-W)1_{[0,M_{r,\epsilon}]^2}
\\
&\qquad \leq \llVert W-W1_{[0,M_{r,\epsilon}]^2} \rrVert +2 \llVert W_j-W
\rrVert _\square,
\end{align*}
showing that for $j$ large enough,
$\|W_j-W_j1_{[0,M_{r,\epsilon}]^2}\|_1\leq\eps r^{-2}/2$.\vadjust{\goodbreak}

We will construct a series of couplings of $H^{(j)}_r $ and $H_r$ by first
coupling $G^{(j)}_r  \sim \mathrm{GPD} (W_j1_{[0,M_{r,\epsilon}]^2},r ) $ and
$G_r
 \sim \mathrm{GPD}  (W1_{[0,M_{r,\epsilon}]^2},r ) $ in
such a way that $\Pr(G^{(j)}_r \neq G_r) \le\epsilon/ 4$ for all sufficiently
large $j$. To see that such couplings exists, we first note that if we define
$\wt W_j=W_j1_{[0,M_{r,\epsilon}]^2}$ and $\wt W=W1_{[0,M_{r,\epsilon }]^2}$,
then $\|\wt W_j-\wt W\|_\square\leq\|W-W_j\|_\square\to0$ as $j\to \infty$.
We can therefore use Lemma~\ref{compact_cut_conv_to_left_conv} to get a
sequence of couplings of $G_r$ and $G^{(j)}_r$ such that for $j$ sufficiently
large, $\Pr(G^{(j)}_r \neq G_r) \le\epsilon/ 4$.

We now observe that given $G^{(j)}_r$, we may sample $H^{(j)}_r$ according
the following scheme:
\begin{enumerate}
\item[1.] Let $(\PP, G_r(W_j1_{U_{r,\epsilon} \times U_{r,\epsilon}}))$ be
    the tuple of the latent point process used to generate a graph, and the
    graph generated by $W_j1_{[0,M_{r,\epsilon}]^2}$ using $\PP$. Draw
    $\PP\given G^{(j)}_r  \sim \Pr(\PP,G_r(W_j1_{[0,M_{r,\epsilon}]^2}))
    \in\cdot \given G_r(W_j1_{[0,M_{r,\epsilon}]^2}) = G^{(j)}_r)$.
\item[2.] Generate a graph $E^{(j)}_r$ according to $W_j1_{\Reals_{+}^2
    \exclude ([0,M_{e,\eps}]^2}$ using $\PP$.
\item[3.] Return the edge set of the graph union of $E^{(j)}_r$ and
    $G^{(j)}_r$ (taking the common vertex set to be $\PP$, and dropping the
    labels).
\end{enumerate}
We define $E_r$ corresponding to $W$ in the obvious way.

Notice that, by construction, the joint distribution of $(\PP, G^{(j)}_r)$ is
the same as the distribution given by drawing $\PP$ as a unit rate Poisson
process and then generating $G^{(j)}_r$ according to $W_j1_{[0,M_\eps]^2}$
using $\PP$. This makes it clear that the sampling scheme reproduces the
distribution given by the Kallenberg representation construction, that is,
$H^{(j)}_r  \sim \mathrm{GPD}  (W_j,r ) $. Also note that $E^{(j)}_r \sim
\mathrm{GPD}  (W_j1_{\Reals_{+}^2 \exclude[0,M_\eps ]^2},r ) $, and $E_r
\sim \mathrm{GPD}  (W1_{\Reals_{+}^2 \exclude[0,M_\eps]^2},r ) $
(marginalizing $G^{(j)}_r$ and $G_r$).

The point of this sampling scheme is that now a coupling of $G^{(j)}_r$ and
$G_r$ immediately lifts to a coupling of $H^{(j)}_r$ and $H_r$ such that
%
\begin{align}
\Pr\bigl(H^{(j)}_r \neq H_r\bigr) &\le\Pr
\bigl(G^{(j)}_r \neq G_r \text{ or } \edges
\bigl(E^{(j)}_r\bigr) > 0 \text { or } \edges(E_r)
> 0 \bigr)
\nonumber
\\
&\le\Pr\bigl(G^{(j)}_r \neq G_r\bigr) + \Pr
\bigl(\edges\bigl(E^{(j)}_r\bigr)>0\bigr) + \Pr \bigl(
\edges(E_r)>0\bigr)
\\
&\le{\eps/4 + \Pr\bigl(\edges\bigl(E^{(j)}_r\bigr)>0\bigr)
+ \Pr\bigl(\edges(E_r)>0\bigr)}. \label{cut_coup_bound}
\end{align}

By Lemma~\ref{small_one_close_dist}, the last two terms of
\eqref{cut_coup_bound} are each at most $\epsilon/2$ and $\eps/4$,
respectively, proving the claim.
\end{proof}

We now turn from the convergence of graphons to convergence of graphs.

\begin{lem}\label{w_wo_pp_equiv}
Let $G_1, G_2, \dots$ be a sequence of graphs such that $\edges(G_j) \to
\infty$ as $j \to\infty$. The following are equivalent:
\begin{enumerate}
\item[1.] The sequence is sampling convergent to $\W$.
\item[2.] The graphon process corresponding to the stretched empirical
    graphon converges to $\W$, in the sense that, for all $r\in\Reals_{+}$,
    $\mathrm{GPD}  (W^{G_j,s},r ) \to\mathrm{GPD}  (\W ,r ) $ weakly as $j
    \to\infty$.
\end{enumerate}
\end{lem}

\begin{proof}
Note that $\loops(G_j) = O(\sqrt{\edges(G_j)})$ is a necessary condition for
convergence in either sense. If $H_{j,r}  \sim \mathrm{GPD} (W^{G_j,s},r ) $
then $H_{r,j}$ may be generated by first sampling $\poiDist
(\frac{r}{\sqrt{2\edges(G_j)}}\vertices (G_j) ) $ vertices with replacement
from $G_j$ and then returning the edge set of the vertex induced subgraph.
The claim is then simply Lemma~\ref{samp_equiv_poi_wr}, the asymptotic
equivalence of this with replacement sampling scheme and
$r/\sqrt{2\edges(G_j)}$-sampling.
\end{proof}

\begin{theorem}\label{cut_implies_sampling}
Let $G_1, G_2, \dots$ be a uniformly tail regular sequence of simple graphs
and let $W$ be some nonrandom graphon. The following are equivalent:
\begin{enumerate}
\item[1.] The sequence converges in stretched cut distance to $W$.
\item[2.] The sequence is sampling convergent to $W$.
\item[3.] The graphon process corresponding to the stretched empirical
    graphon converges to $W$, in the sense that, for all $r\in\Reals_{+}$,
    $\mathrm{GPD}  (W^{G_j,s},r ) \to\mathrm{GPD} (W,r ) $ weakly as $j
    \to\infty$.
\end{enumerate}
\end{theorem}

\begin{proof}
The equivalence of (2) and (3) is a special case of
Lemma~\ref{w_wo_pp_equiv}.

By Lemma~\ref{cut_implies_sampling_lemma} the convergence in stretched cut
distance implies that, almost surely,
\[
\mathrm{GPD} \bigl(W^{G_j,s},r \bigr) \to\mathrm{GPD} (W,r ),
\]
weakly as $j \to\infty$, for all $r\in\Reals_{+}$. Thus (1) implies (3).

Assume the sequence is sampling convergent. Because the sequence is assumed
to be tail regular, it is subsequentially convergent in the stretched cut
distance, by \cite{Borgs:Chayes:Cohn:Holden:2016}, Theorem~15. If there are
two subsequences with distinct limits then, because (1) implies (2), each of
these subsequences will be sampling convergent with the laws of the sampled
graphs given by distinct graphexes. By \cite{Borgs:Chayes:Cohn:Holden:2016},
Theorem~27, graphexes with stretched cut distance not equal to $0$ generate
distinct distributions. Distinct subsequential limits thus contradict the
assumption of sampling convergence, and so (2) implies (1).
\end{proof}

\begin{remark}
Stretched cut convergent graph sequences are always tail regular, so
convergence in stretched cut distance implies sampling convergence without
any need to explicitly check tail regularity.
\end{remark}

\section{Metrization}\label{sec:metrization}

We now translate our main limit result to the language of metric convergence
and give a compactness result.

Recall that a sequence of graphexes $\W_1, \W_2, \dots$ converges in GP to
$\W$ if for all $r\in\Reals_{+}$, $\mathrm{GPD}  (\W _j,r ) \to\mathrm{GPD}
(\W,r ) $ weakly as $j \to\infty$. Let $\delta_{\mathrm{GP}}$ be a
pseudometric on graphexes that metrizes convergence in GP \cite{Janson:2017}.
Then $\delta_{\mathrm{GP}}$ is a proper metric on the space of equivalence
classes of graphexes under the relation that identifies graphexes that
generate the same probability distribution. We will slightly abuse notation
in the case where $\W_j = (W_j, 0, 0)$ and write
$\delta_{\mathrm{GP}}(W_1,W_2) = \delta_{\mathrm{GP}}(\W _1,\W_2)$.

\begin{defn}
Given two finite unlabeled graphs $G,H$, we define\break  $\delta_{\mathrm
{GP}}(G,H) = \delta_{\mathrm{GP}}(W^{G,s},W^{H,s}) +  \lvert1/\edges(G) -
1/\edges(H)  \rvert$.
\end{defn}

The metric $\delta_{\mathrm{GP}}$ on graphs metrizes sampling convergence:
For sequences such that $\edges(G_j) \upto\infty$ (so the limit is a
graphex), this is Lemma~\ref{w_wo_pp_equiv}. For sequences such that
$\edges(G_j) < k$ for some $k < \infty$ for all $j$, this is trivial because
such a sequence is sampling convergent (and $\delta_{\mathrm{GP}}$
convergent) if and only if there is some finite graph $H$ such that, for all
$j$ sufficiently large, $G_j$ is isomorphic to $H$ after excluding isolated
vertices. A~sequence that satisfies neither condition fails to be sampling
convergent and fails to be $\delta _{\mathrm{GP}}$ convergent.

The term $ \lvert1/\edges(G) - 1/\edges(H)  \rvert$ ensures that
$\delta_{\mathrm{GP}} (G,H)=0$ only if $G$ and $H$ are isomorphic after
removing isolated vertices; without this term we would identify complete
bipartite symmetric graphs $K_{n,n}$ for all $n$.

{\spaceskip=0.2em plus 0.05em minus 0.03em For completeness, we also define a
natural metric between graphs and graphexes,} although we do not make
explicit use of it.

\begin{defn}
Given a finite unlabeled graph $G$ and a graphex $\W$, we define
$\delta_{\mathrm{GP}} (G,\W) = \delta_{\mathrm{GP}}(W^{G,s},\W) +
1/\edges(G)$.
\end{defn}

\begin{defn}
Let $\graphspace$ be the metric space of all edge sets of finite graphs
equipped with $\delta_{\mathrm{GP}}$ (identifying $G$ and $H$ whenever
$\delta_{\mathrm{GP}} (G,H)=0$). Also, let $\graphspace_0\subset\graphspace$
be the metric space of all simple graphs in $\graphspace$.
\end{defn}

\begin{defn}
Let $\graphspace^*$ and $\graphspace^*_0$ be the metric completions of
$\graphspace$ and $\graphspace_0$, respectively.
\end{defn}

Our aim is to identify $\graphspace^*$ with a graphex space.

\begin{defn}
Let $\limspace^k$ be the space of equivalence classes of stretched empirical
graphons of $k$ edge graphs, under the equivalence relation $\sim$ defined by
$W_1 \sim W_2$ if and only if $\delta_{\mathrm{GP}}(W_1, W_2) = 0$.

Let $\limspace^\infty$ be the space of equivalence classes of graphexes $\W$
satisfying $\|\W\|_1 \le1$, under the equivalence relation $\sim$ defined by
$\W_1 \sim\W_2$ if and only if $\delta_{\mathrm{GP}}(\W_1, \W_2) = 0$.

Let $\limspace= (\limspace^\infty\times\{0\}) \cup(\bigcup_{k=1}^\infty
\limspace^k \times\{1/k\})$, equipped with the metric $\delta _{\mathrm{GP}}$\break
defined by $\delta_{\mathrm{GP}}((\W_1,p),(\W_2,q)) = \delta_{\mathrm
{GP}}(\W_1,\W_2) +  \lvert p-q  \rvert$.
\end{defn}

The space $\limspace$ is the natural set of limit points of sampling
convergent graph sequences. Splitting the empirical graphons according to the
number of edges of the corresponding graphs allows for an identification with
$\graphspace$.

It is also convenient to define a version of $\limspace$ that excludes loops.
%
\begin{defn}
Let $\limspace_0\subset\limspace$ be the subspace where the graphons have an
a.e. vanishing diagonal (i.e., $W(x,x)=0$ for almost all $x\in\R_+$).
\end{defn}

The next theorem encapsulates two of our results: limits of sampling
convergent sequences are graphexes, and (up to natural equivalencies) all
integrable graphexes arise in this way.

\begin{theorem}\label{theorem:met_space_struct}
$\graphspace^*$ and $\graphspace^*_0$ are isometric to $\limspace$ and
$\limspace_0$, respectively.
\end{theorem}

\begin{proof}
Let $G_1, G_2, \dots$ be a Cauchy sequence in $\graphspace$. If $G_j = H$ for
some graph $H$ and all sufficiently large $j$, then we identify the sequence
with $(W^{H,s},1/\edges(H))$. If $\edges(G_j) \to\infty$ as $j \to \infty$,
Theorem~\ref{lim_is_det_int_graphex} shows that the sampling convergent limit
is identified with some $\W\in\limspace^\infty$. We then identify the
sequence with $(\W,0)$. We have thus defined a map from $\graphspace ^*$ into
$\limspace$.

Suppose $G_1, G_2, \dots$ maps to $(\W_1,p)$ and that $H_1, H_2, \dots$ is a
second Cauchy sequence that maps to $(\W_2,q)$. Then
\[
\begin{aligned} \lim_n \delta_{\mathrm{GP}}(G_n,H_n)
&= \lim_n \delta_{\mathrm
{GP}}\bigl(\bigl(W^{G_n,s},
1/\edges (G_n)\bigr),\bigl(W^{H_n,s}, 1/\edges(H_n)
\bigr)\bigr)
\\
& = \delta_{\mathrm{GP}}\bigl((\W_1,p), (\W_2,q)\bigr),
\end{aligned} %
\]
where the first equality is by definition and the second is by
Lemma~\ref{w_wo_pp_equiv} and the observation that $\delta_{\mathrm{GP}}$
metrizes sampling convergence. The map is thus an isometry.

Finally, the map is surjective: It follows from
Corollary~\ref{every_graphex_lim} that for each $(\W,0) \in\limspace$ there
is some graph sequence with $\W $ as the sampling convergent limit. The
analogous statement for $(W,1/k) \in \limspace$ with $k < \infty$ is
immediate from the definition of $\limspace$.

The fact that under this isometry, $\graphspace^*_0$ gets mapped into
$\limspace_0$ is trivial.
\end{proof}

\begin{theorem}\label{theorem:compact}
If $G_1, G_2, \dots$ in $\graphspace$ is an infinite sequence such that
$\loops(G_j) = O(\sqrt{\edges(G_j)})$, then it has a subsequence that is
convergent in $\graphspace^*$. In particular, the metric completion
$\graphspace_0^*$ of the space of simple graphs equipped with $\delta
_{\mathrm{GP}}$ is compact.
\end{theorem}

\begin{proof}
Let $G_1, G_2, \dots$ be some sequence in $\graphspace$. If there is some $k
\in\Nats$ such that $\sup_j \edges(G_j) + \loops(G_j) < k$, then the
existence of a convergent subsequence is obvious.

It now suffices to show that the closure in $\graphspace^*$ of sequences such
that $\edges(G_j) \to\infty$ and $\loops(G_j) = O(\sqrt{\edges (G_j)})$ is
sequentially compact. By Lemma~\ref{cannon_emb_convs}, it is equivalent to
show that the canonical embeddings of the graph sequence are sequentially
compact in the topology of weak convergence. \cite{Daley:Vere-Jones:2003:v2},
Proposition~11.1.VI, shows that a sufficient condition for uniform tightness
of a family of probability measures on the space of boundedly finite random
measures on $\Reals_{+}^2$, say $(\Pr(\xi_s \in\cdot))_{s\in\mathcal{I}}$, is
that for any bounded Borel set $B$ and any $\epsilon> 0$ there is some
$M\in\Reals_{+}$ such that $\Pr(\xi_s(B) > M) < \epsilon$ for all
$s\in\mathcal{I}$. For a graph sequence $G_1, G_2, \dots$, the canonical
labelings have the property that $\EE[\mathsf{Lbl}(G_j)([0,r]^2)] \le{r^2} +
r \loops(G_j)/\sqrt{2\edges(G_j)}$ (with equality whenever
$\sqrt{2\edges(G_j)} > r$), from which the uniform tightness condition
follows trivially. The result then follows by Prokhorov's theorem.
\end{proof}

\section{Sampling defines exchangeable random graphs}\label
{sec:samp_defines_GP}

The {time} parameter of a graphex process is related to $p$-sampling by the
observation that if $G  \sim \mathrm{GPD}  (\W,s ) $ then $\mathsf{Smpl}
(G,p)
 \sim
\mathrm{GPD}  (\W,ps ) $. That is, the relationship between graphs at
different {time}s is captured by $p$-sampling. In this section, we show that
this is in fact a defining property of sparse exchangeable random graphs.

\begin{defn}
Call $(G_s)_{s\in\Reals_{+}}$ an \defnphrase{unlabeled random graph process}
indexed by $\Reals_{+}$ if, for all $s$, $G_s$ is a finite unlabeled graph,
and, for all $s \le t$, $G_s \subseteq G_t$ in the sense that there is some
subgraph of $G_t$ that is isomorphic to $G_s$.
\end{defn}

\begin{theorem}\label{psamp_defining}
Let $(G_s)_{s\in\Reals_{+}}$ be an unlabeled random graph process such that
$\edges(G_s) \upto\infty\as$ as $s\to\infty$. For each $s\in \Reals_{+}$ and
$p\in(0,1)$, let $\mathsf{Smpl} (G_s,p)$ be a $p$-sampling of $G_s$. If for
all $s\in\Reals_{+}$ and $p\in(0,1)$,
\[
\mathsf{Smpl} (G_s,p) \equaldist G_{p s},
\]
then there is some (possibly random, possibly nonintegrable) almost surely
nonzero graphex $\W$ such that, for all $s \in\Reals_{+}$, $G_s \given \W
 \sim \mathrm{GPD}  (\W,s ) $.
\end{theorem}

\begin{proof}
To establish the claimed result, it obviously suffices to show that there is
some $\W$ such that $\mathsf{Lbl}_{s}(G_s) \equaldist\KEG_s$, where
$(\KEG_s)_{s\in\Reals_{+}}$ is a graphex process generated by $\W$.

Let $r,s \in\Reals_{+}$ be such that $r < s$. Then
%
\begin{equation}
\mathsf{Lbl}_{s}(G_s) \bigl([0,r)^2
\cap\cdot\bigr) \equaldist\mathsf{Lbl}_{r}\biggl(\mathsf{Smpl}
\biggl({G_s}, \frac{r}{s}\biggr)\biggr) \equaldist
\mathsf{Lbl}_{r}({G_r}).\label{lim_well_def}
\end{equation}
The first equality follows by the observation that each vertex of
$\mathsf{Lbl}_{s}(G_s)$ has label in $[0,r)$ independently with probability
$r/s$, so that $\mathsf{Lbl}_{s}(G_s)$ restricted to $[0,r)^2$ has the same
distribution as $\mathsf{Lbl}_{r}(\mathsf{Smpl} ({G_s},r/{s}))$. The second
equality is by hypothesis.

Let $\xi$ be a point process with distribution defined by, for any bounded
Borel sets $B_1, \dots, B_n \subseteq\Reals_{+}^2$,
\[
\bigl\{\xi(B_1), \dots, \xi(B_n)\bigr\} \equaldist\lim
_{s\to\infty}\bigl\{ \mathsf{Lbl}_{s}(G_s)
(B_1), \dots, \mathsf {Lbl}_{s}(G_s)
(B_n) \bigr\}.
\]
Equation \eqref{lim_well_def} makes it clear that the limiting distribution
on the right-hand side is well-defined. Moreover, using the fact that the
joint distribution is defined as counts of the random labeling point process,
the consistency conditions necessary for $\lim_{s\to\infty}\{
\mathsf{Lbl}_{s}(G_s)(B_1), \dots, \mathsf {Lbl}_{s}(G_s)(B_n) \}$ to be
counts with respect to some point process can easily be seen to be satisfied.
By the Kolmogorov existence theorem for point processes (see
\cite{Daley:Vere-Jones:2003:v2}, Theorem~9.2.X), this suffices to show that
$\xi$ exists and has a well-defined distribution. Also note that $\xi$ is
purely atomic by construction.

Observe that by \eqref{lim_well_def} and the definition of $\xi$ it holds
that, for all $r \in\Reals_{+}$,
%
\begin{equation}
\mathsf{Lbl}_{r}(G_r) \equaldist\xi\bigl([0,r
)^2 \cap \cdot\bigr). \label
{eq:lim_trunc_is_lbl}
\end{equation}
In consequence, for any measure-preserving transformation $\phi$ on $[0,r)$,
$\xi\circ(\phi\otimes\phi) \equaldist\xi$. In particular then, for any dyadic
partitioning of $\Reals_{+}$ and any transposition $\tau$ of this dyadic
partitioning we may take $r$ large enough such that the transposition acts
only in $[0,r)$, and thus $\xi\circ(\tau\otimes\tau) \equaldist \xi$. By
\cite{Kallenberg:2005}, Proposition~9.1, this implies that $\xi$ is
exchangeable.

We now have that $\xi$ is a purely atomic exchangeable point process, so by
the Kallenberg representation theorem,
Theorem~\ref{theorem:graphex_rep_theorem}, there is some graphex $\W$ such
that $\xi$ is generated by $\W$. The proof is then completed by again
invoking~\eqref{eq:lim_trunc_is_lbl}.
\end{proof}


\section*{Acknowledgments}
The authors thank Daniel Roy for many helpful discussions, Svante Janson for
many valuable comments and suggestions on an earlier draft and the anonymous
referee for detailed and helpful comments.


%

\end{document}